\documentclass[12pt]{amsart}
\pdfoutput=1  % This helps with some issues with arXiv submissions.
\usepackage[latin1]{inputenc}
\usepackage[french,english]{babel} 
\usepackage{amsmath,amsfonts,amssymb,mathtools,amsthm}

\usepackage[
  hmargin=3.62cm,
  vmargin=3.5cm,
]{geometry}

\usepackage[colorlinks, 
            linkcolor=MidnightBlue, 
            citecolor=MidnightBlue,
            urlcolor =BlueViolet,
           ]{hyperref}
\usepackage[usenames,dvipsnames]{xcolor} % For colored text.
\usepackage[all]{xy}
\usepackage{colonequals}                 % For a clean := symbol.

\newtheorem{theorem}{Theorem}[section]
\newtheorem{proposition}[theorem]{Proposition}
\newtheorem{lemma}[theorem]{Lemma}
\newtheorem{corollary}[theorem]{Corollary}

\theoremstyle{remark}
\newtheorem{remark}[theorem]{Remark}
\newtheorem{example}[theorem]{Example}
\newtheorem{definition}[theorem]{Definition}

\numberwithin{equation}{section} 

\DeclareMathOperator{\Aut}{Aut}
\DeclareMathOperator{\End}{End}
\DeclareMathOperator{\Hom}{Hom}
\DeclareMathOperator{\Jac}{Jac}
\DeclareMathOperator{\Pic}{Pic}
\DeclareMathOperator{\PGL}{PGL}
\DeclareMathOperator{\Tr}{Tr}
\DeclareMathOperator{\Reg}{Reg}

\newcommand\CC{{\mathbb C}}
\newcommand\FF{{\mathbb F}}
\newcommand\PP{{\mathbb P}}
\newcommand\QQ{{\mathbb Q}}

\newcommand\ZZ{{\mathbb Z}}

\newcommand{\frakA}{{\mathfrak A}}

\newcommand{\frakp}{{\mathfrak p}}

\newcommand{\calI}{{\mathcal I}}
\newcommand{\calJ}{{\mathcal J}}
\newcommand{\calO}{{\mathcal O}}
\newcommand{\calOplus}{{\mathcal O}^{+}}

\newcommand{\pibar}{\overline{\pi}}
\newcommand{\Fbar}{\overline{\FF}}

\newcommand{\Ahat}{\widehat{A}}
\newcommand{\Ehat}{\widehat{E}}
\newcommand{\Jhat}{\widehat{J}}
\newcommand{\alphahat}{\widehat{\alpha}}
\newcommand{\betahat}{\widehat{\beta}}
\newcommand{\varphistarhat}{\widehat{\varphi^*}}
\newcommand{\psihat}{\widehat{\psi}}

\newcommand\Fq{\FF_q}
\newcommand\Fqtwo{\FF_{q^2}}
\newcommand{\Kplus}{K^{+}}

\newcommand{\Uplus}{U^{+}}

\makeatletter
\@namedef{subjclassname@2020}{%
  \textup{2020} Mathematics Subject Classification}
\makeatother

\title[Deducing information about curves]
      {Deducing information about curves over finite fields from their Weil polynomials}
\date{26 October 2022}

\author{Everett W. Howe}
\address{Unaffiliated mathematician, 
         San Diego, CA 92104, USA}
\email{\href{mailto:however@alumni.caltech.edu}{however@alumni.caltech.edu}}
\urladdr{\href{http://ewhowe.com}{http://ewhowe.com}}

\keywords{Abelian variety, Jacobian variety, curve, isogeny, finite field}

\subjclass[2020]{Primary 11G20; Secondary 11G10, 14G10, 14G15, 14H40, 14K02, 14K15}

\begin{document}

\begin{abstract}
We discuss methods for using the Weil polynomial of an isogeny class of abelian 
varieties over a finite field to determine properties of the curves (if any)
whose Jacobians lie in the isogeny class. Some methods are strong enough to show
that there are \emph{no} curves with the given Weil polynomial, while other 
methods can sometimes be used to show that a curve with the given Weil 
polynomial must have nontrivial automorphisms, or must come provided with a map 
of known degree to an elliptic curve with known trace. Such properties can 
sometimes lead to efficient methods for searching for curves with the given Weil
polynomial.

Many of the techniques we discuss were inspired by methods that Serre used in
his 1985 Harvard class on rational points on curves over finite fields. The 
recent publication of the notes for this course gives an incentive for reviewing
the developments in the field that have occurred over the intervening years.
\end{abstract}

\maketitle

\tableofcontents

%%%%%%%%%%%%%%%%%%%%%%%%%%%%%%%%%%%%%%%%%%%%%%%%%%%%%%%%%%%%%%%%%%%%%%%%%%%%%%%%
%%%%%%%%%%%%%%%%%%%%%%%%%%%%%%%%%%%%%%%%%%%%%%%%%%%%%%%%%%%%%%%%%%%%%%%%%%%%%%%%
\section{Introduction}
\label{sec:intro}

In his 1985 Harvard course \emph{Rational points on curves over finite fields},
Serre introduced many ideas for studying curves over finite fields with a given
number of points. In the succeeding decades, many of these ideas have been 
developed further and have led to interesting new results and techniques. In 
this chapter we will consider the historical development of one of these ideas: 
determining the existence or nonexistence of a curve over a finite field with a
given number of points by using properties of the possible values of its 
\emph{Weil polynomial}, which is the characteristic polynomial of Frobenius for 
the curve's Jacobian. Let us begin by reviewing how this idea shows up in the
recently published notes for Serre's course~\cite{Serre2020}.

The Weil polynomial $f$ of a curve of genus $g$ over $\Fq$ is a monic polynomial
of degree~$2g$ with integer coefficients. The Weil conjectures for abelian
varieties over finite fields \cite[\S19]{Milne1986a} show that all of the
complex roots of $f$ lie on the circle of radius $\sqrt{q}$ in~$\CC$, and 
moreover there is a monic polynomial $h\in\ZZ[x]$ of degree $g$ such that 
$f(x) =  x^g h(x + q/x)$; this polynomial $h$ is known as the 
\emph{real Weil polynomial} of the curve, and all of its complex roots are
real numbers that lie in the interval $[-2\sqrt{q},2\sqrt{q}]$.

Serre's generalization~\cite{Serre1983} of the Weil bound on the number of
points on a genus-$g$ curve $C$ over $\Fq$ says that $\#C(\Fq)\le q + 1 + gm$,
where $m = \lfloor 2\sqrt{q} \rfloor$. The \emph{defect} of $C$ is the
difference between this upper bound and the actual value of $\#C(\Fq)$. Let $t$
be the trace of the Frobenius of~$C$; this is the negative of the coefficient of
$x^{2g-1}$ in~$f$, and the negative of the coefficient of $x^{g-1}$ in~$h$. Then
$\#C(\Fq) = q + 1 - t$, so the defect of $C$ is~$gm+t$. Serre's goal in his 
Harvard course was to investigate how close the number of points on a curve (of
a given genus, over a given finite field) can come to the Weil--Serre 
bound --- that is, he hoped to determine how small a curve's defect can be. For 
very small values of the defect, Serre could give an explicit list of 
possibilities for the real Weil polynomial of a curve with that defect.

For example, Serre showed~\cite[{\S}II.2]{Serre2020} that if a curve has 
defect~$0$, then its real Weil polynomial must be $(x + m)^g$; if a curve has 
defect~$1$, then its real Weil polynomial must be either
\[
(x + m)^{g-1}(x + m - 1) \quad\text{or}\quad 
(x + m)^{g-2}((x+m)^2 - (x + m) - 1);
\]
and if a curve has defect~$2$, then its real Weil polynomial is equal to
$H(x + m)$, where $H$ is one of the following seven polynomials:
\begin{align*}
&x^{g-1}(x - 2)\\
&x^{g-2}(x - 1)^2\\
&x^{g-2}(x^2 - 2x - 1)\\
&x^{g-2}(x^2 - 2x - 2)\\
&x^{g-3}(x - 1)(x^2 - x - 1)\\
&x^{g-3}(x^3 - 2x^2 - x + 1)\\
&x^{g-4}(x^2 - x - 1)^2.
\end{align*}
But Serre goes on to show that some of these polynomials are equal to real Weil
polynomials of abelian varieties over $\Fq$ only under certain conditions 
on~$q$. Furthermore, some of these polynomials \emph{never} come from curves;
for instance, ${(x + m)^{g-1}}{(x + m - 1)}$ is never the real Weil polynomial 
of a curve of genus~${g > 1}$.

Serre obtained these lists of possible real Weil polynomials by using a result 
of Smyth~\cite{Smyth1984a} that gives a complete list of the totally positive 
algebraic integers $\alpha$ such that the difference between the trace of 
$\alpha$ and the degree of $\alpha$ is at most~$6$. (Serre could therefore have
extended his lists up to defect~$6$, but even for defect $4$ the list would have
been quite long.) In Section~\ref{sec:enumerate}, we review work that shows that
adaptations of Smyth's techniques can be used to generate a list of real Weil
polynomials of abelian varieties that necessarily contains the real Weil
polynomials of curves of a given genus and defect over a specific field~$\Fq$. 
Such a list can be computed in a reasonable amount of time even for quite large
values of the defect, provided that the fractional part of $2\sqrt{q}$ is small.
For example, an implementation of a variant of Smyth's algorithm on a laptop 
computer produces in just a couple of seconds a list of all real Weil 
polynomials of $18$-dimensional abelian varieties over $\FF_9$ that might 
contain the Jacobian of a curve of defect~$33$.

Serre's technique for eliminating ${(x + m)^{g-1}}{(x + m - 1)}$ as a possible
real Weil polynomial of a curve relies on a technique that we call the
``resultant~1" method. In Section~\ref{sec:nojac} we review this method, as
well as other methods for that have been developed over the years for showing 
that certain real Weil polynomials cannot come from curves.

Sometimes we find that none of the presently known techniques eliminates the 
possibility that a given real Weil polynomial comes from a curve --- but these 
techniques  may provide useful information about curves with that real Weil 
polynomial that can be used to further analyze the situation. In 
Section~\ref{sec:deducing} we look at the type of information we can sometimes
obtain, and we show how this extra information can help us implement an 
efficient search for curves with the given real Weil polynomial.

For $g = 1$ and $g = 2$ we know precisely which isogeny classes of 
$g$-dimensional abelian varieties over $\Fq$ contain Jacobians. For genus $1$
this is trivial, and for genus $2$ the result is given in a paper of Howe, Nart,
and Ritzenthaler~\cite[Theorem~1.2, p.~240]{HoweNartEtAl2009}. Even for genus 
$3$ the question becomes much more difficult; 
Ritzenthaler~\cite{Ritzenthaler2010} discusses some of the complications. In 
this chapter we hope to communicate the difficulty, in higher genera, of 
answering even the very specific question of whether a given isogeny class over
a given finite field contains a Jacobian, let alone the question of classifying
all such isogeny classes.

Serre's course was motivated by the problem of computing~$N_q(g)$, the maximum
number of points on a curve of genus $g$ over~$\Fq$. We will not be directly
concerned with $N_q(g)$ in this chapter, but when we give examples or
applications below we mostly choose them from papers in which bounds on $N_q(g)$
were computed for specific $q$ and~$g$.

\subsubsection*{Acknowledgments}
I gratefully take this opportunity to thank Alp Bassa, Elisa Lorenzo 
García, Christophe Ritzenthaler, and René Schoof for instigating the
project to \TeX\ the course notes for Serre's class that were taken by Fernando
Gouvêa back in 1985, and Joan-Carles Lario for joining them on the 
scientific committee for the conference celebrating the release of the published
notes. Others had tried to start such a project in the past, but Alp, Elisa, 
Christophe, and René managed to corral volunteers, edit the resulting
crowdsourced file, and work with Serre and with the Société 
mathématique de France to create a valuable resource for researchers. Thanks
also to Fernando Gouvêa for his thorough notes, which for thirty-five years 
spread like samizdat throughout the community of arithmetic geometers; to the
more than thirty volunteers who {\TeX}ed and proofread the text; and of course 
to Serre himself, for creating the course so long ago and for editing and 
updating the material for publication. My thesis work involved abelian varieties
over finite fields, but Serre's 1985 Harvard notes inspired a turn to curves and
Jacobians over finite fields that has been a major theme of my mathematical work
ever since. It was an honor to make contributions to the published version of
the course notes, and I thank the organizers for inviting me to speak at the 
conference and to write the present chapter. Finally, I am grateful to the 
referee, whose comments and thorough reading helped improved the exposition of 
this chapter.

%%%%%%%%%%%%%%%%%%%%%%%%%%%%%%%%%%%%%%%%%%%%%%%%%%%%%%%%%%%%%%%%%%%%%%%%%%%%%%%%
%%%%%%%%%%%%%%%%%%%%%%%%%%%%%%%%%%%%%%%%%%%%%%%%%%%%%%%%%%%%%%%%%%%%%%%%%%%%%%%%
\section{Enumerating isogeny classes}
\label{sec:enumerate}

In this section we discuss methods of enumerating the real Weil polynomials of 
isogeny classes of $g$-dimensional abelian varieties over $\Fq$ that might
possibly contain Jacobians of curves. In Section~\ref{ssec:smyth} we review the 
result of Smyth that Serre used to calculate the real Weil polynomials of curves
with very small defect. Smyth's work involves enumerating totally positive 
algebraic integers for which the difference between the trace and the degree is 
small. In Section~\ref{ssec:realweil} we discuss how to apply Smyth's 
computational techniques directly to the problem of finding the isogeny classes
of curves with small defect, without explicitly taking the intermediate step of 
considering totally positive algebraic integers.

%%%%%%%%%%%%%%%%%%%%%%%%%%%%%%%%%%%%%%%%%%%%%%%%%%%%%%%%%%%%%%%%%%%%%%%%%%%%%%%%
\subsection{Totally positive algebraic integers with small trace}
\label{ssec:smyth}

Serre's enumeration of the possible real Weil polynomials with a given 
defect~\cite[{\S}II.2]{Serre2020} relies on work of Smyth~\cite{Smyth1984a} 
concerning certain totally positive algebraic integers. The connection between 
the two topics is the following.

Suppose $h$ is the real Weil polynomial of a genus-$g$ curve over $\Fq$ with 
defect~$d$. Then $h\in\ZZ[x]$ is monic and has degree~$g$, and the trace $t$ of
$h$ (that is, $-1$ times the coefficient of $x^{g-1}$ in $h$) is equal to 
$d - gm$, where $m = \lfloor 2\sqrt{q}\rfloor$. As we have already noted, all of
the complex roots of $h$ are real and lie in the closed interval
$[-2\sqrt{q}, 2\sqrt{q}]$. It follows that if we set $p(x) = h(x - m - 1)$ then 
all the roots of $p$ are real and lie in the positive interval 
$[m + 1 -2\sqrt{q}, m + 1 + 2\sqrt{q}]$. Note that the trace of $p$ is equal to 
$t + g(m+1) = d + g$. If $p$ is irreducible, then it defines a totally positive
algebraic integer $\alpha$ whose trace minus its degree is equal 
to~$(d+g) - g = d$.

Smyth's paper considers exactly the problem of finding all totally positive 
algebraic integers $\alpha$ such that the trace of $\alpha$ minus its degree is
equal to a given non-negative integer~$d$. The first step to solving this 
problem comes from a result of Siegel~\cite[Theorem~III, p.~303]{Siegel1945},
which states that if $\alpha$ is a totally positive algebraic integer whose
minimal polynomial is neither $x-1$ nor $x^2 - 3x + 1$, then the trace of
$\alpha$ is greater than three-halves its degree. Thus, if $\alpha$ is a totally
positive algebraic integer whose trace minus degree is equal to~$d\ge0$, then
the degree of $\alpha$ is strictly less than~$2d$, unless $\alpha$ is one of the
exceptions listed above. In fact, using some of his own earlier
work~\cite{Smyth1984b}, Smyth improves Siegel's lower bound: He shows that apart
from a few explicit exceptions, Siegel's three-halves can be replaced with
$1.7719$. With this improvement, the bound on the degree is lowered from $2d$ to
something slightly better than~$1.3 d$.

These upper bounds on the degree allow Smyth to reduce his problem to that of 
finding those monic irreducible polynomials in $\ZZ[x]$ of a given degree $n$
and a given trace $t$ whose complex roots are all positive real numbers. The 
technique Smyth uses to enumerate such polynomials goes back at least to 
Robinson~\cite{Robinson1964}. The basic idea is that if $p$ is such a 
polynomial, then the roots of its derivative $p'$ are also all positive real 
numbers, and likewise for the higher derivatives of $p$ as well. If we write
\[ 
p = x^n + b_1 x^{n-1} + \cdots + b_n,
\]
with $b_1 = -t$, then
\[ 
p^{(n-2)} = (n!/2)\, x^2 - (n-1)!\, t \, x + (n-2)! \,b_2.
\]
The condition that $p^{(n-2)}$ have $2$ positive real roots limits the possible
values of $b_2$ to a computable interval; one can visualize this by picturing
the graph of 
\[
y = (n!/2)\, x^2 - (n-1)!\, t\, x
\]
and imagining how the number of positive real roots changes when one raises or 
lowers the graph by adding a constant to the right-hand side. For each allowable
value of~$b_2$, one finds in the same manner a finite number of possibilities 
for $b_3$ by considering $p^{(n-3)}$. This leads to a recursive procedure for 
finding all polynomials in $\ZZ[x]$ of degree $n$ and trace $t$ whose complex
roots are all positive real numbers.

%%%%%%%%%%%%%%%%%%%%%%%%%%%%%%%%%%%%%%%%%%%%%%%%%%%%%%%%%%%%%%%%%%%%%%%%%%%%%%%%
\subsection{Real Weil polynomials of Jacobians}
\label{ssec:realweil}

In the previous section we showed how to reduce the problem of finding real Weil
polynomials of a given degree $g$ and defect $d$ to the problem of enumerating 
totally positive algebraic integers of degree at most $g$ and with degree minus
trace at most~$d$, as studied by Smyth. When working over a specific finite
field, however, we can also refrain from making this reduction, and instead
simply adapt the techniques used by Smyth to our specific situation. Part of
Serre's argument for showing that $N_2(7) = 10$ involves this 
idea~\cite[{\S}7.2]{Serre2020}, but the strategy was first fully described and 
employed by Lauter~\cite[\S 2]{Lauter2000}; we describe her method here.

Given a prime power~$q$, we let $m = \lfloor 2\sqrt{q}\rfloor$. As we have seen,
the real Weil polynomial $h$ of a genus-$g$ curve over $\Fq$ with defect $d$ can
be written
\[
h = x^g + b_1 x^{g-1} + \cdots + b_g
\]
where $b_1 = gm - d$, and all of the complex roots of $h$ are real and lie in
the interval $[-2\sqrt{q}, 2\sqrt{q}]$. Then all of the roots of every 
derivative of $h$ also lie in this interval, and the method used by Smyth gives
us a straightforward way of enumerating all of the degree-$g$ polynomials $h$ 
with initial coefficient $gm-d$ and with all roots in the given interval.

There is another fact that we can use, however, to cut down the space of
polynomials that we want to enumerate. Let $f$ be the Weil polynomial 
corresponding to~$h$, so that $f(x) = x^g h(x + q/x)$. The roots $\alpha_i$ of 
$f$ are the eigenvalues of Frobenius for the curve~$C$, and for every integer 
$n>0$ we have
\[
\#C(\FF_{q^n}) = q^n + 1 - \sum_{i=1}^{2g} \alpha_i^n.
\]
These point counts of $C$ over extension fields are of course related to the
number $a_n$ of degree-$n$ places on $C$; namely, we have
\[
\#C(\FF_{q^n}) = \sum_{d\mid n} d a_d,
\]
and by Möbius inversion~\cite[\S2.7]{Apostol1976} we find that
\[
a_n =\frac{1}{n} \sum_{d\mid n} \mu\Bigl(\frac{n}{d}\Bigr) \#C(\FF_{q^d}),
\]
where $\mu$ is the multiplicative arithmetic function that takes the value $-1$ 
on primes and $0$ on proper powers of primes.

\begin{lemma}
\label{L:countsandcoefs}
For each $n = 1, \ldots, g$, the sequence of coefficients $(b_1, \ldots, b_n)$
determines and is determined by the sequence of place counts 
$(a_1, \ldots, a_n)$.
\end{lemma}

\begin{proof}
Write the Weil polynomial of $C$ as
\[
f = x^{2g} + c_1 x^{2g-1} + \cdots + c_{g-1} x^{g+1} + c_g x^g 
    + q c_{g-1} x^{g-1} + \cdots + q^{g-1} c_1 x + q^{g}.
\]
By expanding out the relation $f(x) = x^g h(x + q/x)$ that connects the Weil
polynomial to the real Weil polynomial, we find that for each $n\le g$ the 
sequence $(b_1, \ldots, b_n)$ determines and is determined by the sequence 
$(c_1,\ldots,c_n)$. The coefficient $c_i$ is $(-1)^i$ times the $i$\up{th} 
symmetric function in the $\alpha_j$, so the Girard--Newton relations
(\cite[p.~A~IV.65]{Bourbaki1981}, \cite{Funkhouser1930}) show that the sequence 
$(c_1,\ldots,c_n)$ determines and is determined by the first $n$ power sums of 
the $\alpha_j$, and hence also by the points counts of $C$ over the first $n$
extensions of~$\Fq$. Finally, the point counts of $C$ over the first $n$ 
extensions of $\Fq$ determine and are determined by the sequence 
$(a_1, \ldots, a_n)$ of place counts for~$C$.
\end{proof}

Lauter observes that once we have chosen $(b_1, \ldots, b_{n-1})$, the condition
that $a_i\ge 0$ for every $i>0$ can be used to cut down the size of the
intervals that we let $b_n$ range over. Lauter's algorithm was implemented as
part of the Magma programs accompanying the second paper of Howe and
Lauter~\cite{HoweLauter2012} devoted to improving bounds on~$N_q(g)$; these 
programs can be found on the web site of the author of this chapter. 

McKee and Smyth~\cite[\S 3.2]{McKeeSmyth2004} describe an improvement to Smyth's
algorithm~\cite{Smyth1984a} for computing totally positive algebraic integers 
with a given degree and trace. They use this new algorithm to improve Siegel's
lower bound even further: They show that apart from a short explicit list of 
exceptions, a totally positive algebraic integer $\alpha$ must satisfy 
$\Tr\alpha > 1.778378 \cdot \deg\alpha$. As far as the author is aware, McKee 
and Smyth's improvement has not been incorporated into any algorithms for
enumerating isogeny classes.

Kedlaya~\cite{Kedlaya2008} gives an algorithm for computing the Weil polynomials
of all abelian varieties of a given dimension over a given finite field, 
following the basic strategy of Robinson and Smyth but including many new 
techniques and improvements. His algorithm was used for the portion of the
$L$-functions and Modular Forms Database~\cite{LMFDB} that includes information
on abelian varieties of small dimension over certain finite fields.

%%%%%%%%%%%%%%%%%%%%%%%%%%%%%%%%%%%%%%%%%%%%%%%%%%%%%%%%%%%%%%%%%%%%%%%%%%%%%%%%
%%%%%%%%%%%%%%%%%%%%%%%%%%%%%%%%%%%%%%%%%%%%%%%%%%%%%%%%%%%%%%%%%%%%%%%%%%%%%%%%
\section{Showing there is no Jacobian in an isogeny class}
\label{sec:nojac}

Perhaps the easiest way to show that there is no Jacobian in an isogeny class is
to check whether the associated Weil polynomial would predict a negative number 
of degree-$d$ places on a curve with that Weil polynomial, for some~${d>0}$. As
indicated in the previous section, we can easily incorporate that test into our
algorithm for generating Weil polynomials. In this section, we consider more 
sophisticated tests.

In Section~\ref{ssec:ppavs} we review work that can sometimes be used to
determine whether or not an isogeny class of abelian varieties contains a 
variety that has a principal polarization. In 
Section~\ref{ssec:noindecomposable} we introduce the idea of checking whether an
isogeny class contains any abelian varieties that have an \emph{indecomposable}
principal polarization, which is the theme of the following four sections. In
Section~\ref{ssec:resultant1} we describe Serre's ``resultant~1" method, and 
place it into a much more general context that will reappear in 
Section~\ref{sec:deducing}. In Section~\ref{ssec:hermitian} we review the idea 
of Hermitian modules, and show how such modules can be used to prove that 
certain powers of elliptic curves are not isogenous to varieties with 
indecomposable principal polarizations. In Section~\ref{ssec:counting} we show 
that sometimes we can count the number of abelian varieties with indecomposable 
principal polarizations that lie in a given isogeny class, and that sometimes 
this shows that none exist. We close with Section~\ref{ssec:supersingular}, 
where we show that there are strong restrictions on the Weil polynomials of 
varieties with indecomposable principal polarizations if they are isogenous to
an ordinary variety times a power of a supersingular elliptic curve.

%%%%%%%%%%%%%%%%%%%%%%%%%%%%%%%%%%%%%%%%%%%%%%%%%%%%%%%%%%%%%%%%%%%%%%%%%%%%%%%%
\subsection{Isogeny classes with no principally polarized varieties}
\label{ssec:ppavs}

As Serre summarizes in his course notes~\cite[\S III.5]{Serre2020}, a 
\emph{polarization} of an abelian variety $A$ is an isogeny from $A$ to its dual
variety $\Ahat$ that has certain symmetry and positivity 
properties~\cite[\S 13]{Milne1986a}; also, the Jacobian variety of a curve is in
particular an abelian variety that comes provided with a \emph{principal}
polarization, which is a polarization whose degree as an isogeny 
is~$1$~\cite[\S 6]{Milne1986b}. One way to show that an isogeny class of abelian
varieties does not contain a Jacobian, then, is to show that the isogeny class 
does not contain any varieties that have a principal polarization.

A $g$-dimensional abelian variety over~$\Fq$ is \emph{ordinary} if the 
coefficient of $x^g$ in its Weil polynomial is coprime to~$q$. (This is 
equivalent to the constant term of its real Weil polynomial being coprime 
to~$q$.) One can always determine whether an isogeny class of ordinary abelian 
varieties contains a principally polarized\footnote{
  To be technically correct, we should speak of whether an isogeny class 
  contains a principally polariz\emph{able} variety, but in practice people 
  are rarely so precise in their language.}
variety~\cite[Theorem~1.1, p.~2362]{Howe1995}. Also, there are easily satisfied 
conditions that imply that an isogeny class --- ordinary or not --- \emph{does}
contain a principally polarized variety. For example, suppose $A$ is a simple
$g$-dimensional abelian variety over $\Fq$ and let $\calI$ be its isogeny class.
Let $\pi\in\End A$ be the Frobenius endomorphism. Then 
$K\colonequals\QQ(\pi)\subseteq(\End A)\otimes\QQ$ is a number field, which is 
either totally real or a \emph{complex multiplication} (\emph{CM})~field, that 
is, a totally imaginary quadratic extension of a totally real field~$\Kplus$. 
When $A$ is ordinary, $K$ is a CM~field. 

\begin{theorem}
\label{T:PPAV}
\quad
\begin{itemize}
\item[(a)] If $K$ is totally real, then $\calI$ contains a principally polarized
           variety.
\item[(b)] Suppose $K$ is a CM~field. If $K$ is ramified over $\Kplus$ at a
           finite prime, or if $\pi-\pibar$ is divisible by a prime of $\Kplus$
           that is inert in $K/\Kplus$, then $\calI$ contains a principally 
           polarized variety. 
\item[(c)] Suppose $\calI$ is ordinary and $K$ does not satisfy the conditions 
           in part~\textup{(b)}. Then ${N_{K/\QQ}(\pi-\pibar)}$ is a square. Let
           $s$ be its positive square root, and let $c_g$ be the coefficient of 
           $x^g$ in the Weil polynomial for~$A$. Then $\calI$ contains a 
           principally polarized variety if and only if $c_g\equiv s\bmod m$,
           where $m = q$ if $q>2$, and $m = 4$ if~$q = 2$.
\end{itemize}
\end{theorem}

\begin{remark}
The coefficient $c_g$ of $x^g$ in the Weil polynomial for $A$ is congruent 
modulo $2q$ to the constant term of the real Weil polynomial for~$A$, and since
the integer $m$ in part~(c) divides~$2q$, we can check the condition in part~(c)
using this constant term instead of~$c_g$.
\end{remark}

\begin{proof}[Proof of Theorem~\ref{T:PPAV}]
This is a combination of \cite[Theorem~1.1, p.~584]{Howe1996},
\cite[Corollary~11.4, p.~2391]{Howe1995},
and~\cite[Proposition~11.5, pp.~2391--92]{Howe1995}.
\end{proof}

\begin{corollary}
\label{C:odd}
Every simple odd-dimensional abelian variety over a finite field is isogenous
to a principally polarized variety.
\end{corollary}

\begin{proof}
This is \cite[Theorem~1.2, p.~584]{Howe1996}, and it follows from parts~(a) 
and~(b) of Theorem~\ref{T:PPAV} combined with the fact that a CM~field is
ramified over its real subfield at a finite prime if the real subfield has odd
degree over~$\QQ$~\cite[Lemma~10.2, p.~2385]{Howe1995}.
\end{proof}

\begin{example}
Consider the isogeny class $\calI$ of abelian varieties over $\FF_8$ with real
Weil polynomial $h \colonequals x^4 + 13 x^3 + 58 x^2 + 102 x + 57$. A curve
with this real Weil polynomial would have a non-negative number of places of 
every degree, so there is no initial reason to think that $\calI$ could not
contain a Jacobian. Let us see whether we can apply Theorem~\ref{T:PPAV}.

The constant term of $h$ is odd, so $\calI$ is an ordinary isogeny class. Let
$f \colonequals x^4 h(x + 8/x)$ be the Weil polynomial corresponding to~$h$. We
check that 
\[ 
f = x^8 + 13 x^7 + 90 x^6 + 414 x^5 + 1369 x^4 
        + 3312 x^3 + 5760 x^2 + 6656 x + 4096
\]
and that $f$ is irreducible. Let $K$ be the number field defined by~$f$, let 
$\pi$ be a root of $f$ in~$K$, and let $\Kplus$ be the maximal real subfield 
of~$K$. The field $K$ contains a primitive cube root of unity, so in fact we
have $K = \Kplus(\sqrt{-3})$.

We check that the discriminant of $\Kplus$ is $3^2\cdot 1371$ and the
discriminant of $K$ is $3^4\cdot 1371^2$, so $K$ is unramified over $\Kplus$ at
all finite primes. The norm of $\pi-\pibar$ from $K$ to $\QQ$ is~$199^2$, and 
since $-3$ is a square modulo~$199$, every prime of $\Kplus$ that divides 
$\pi-\pibar$ splits in $K/\Kplus$. The coefficient of $x^4$ in $f$ is~$1369$, 
which is $1$ modulo~$8$. Since the positive square root of 
$N_{K/\QQ}(\pi-\pibar)$ is~$199$, which is $7$ modulo~$8$, Theorem~\ref{T:PPAV}
shows that $\calI$ does not contain any principally polarized varieties.

Without using Theorem~\ref{T:PPAV}, it is not clear how one would show that 
there are no Jacobians in~$\calI$. One could conceivably enumerate all genus-$4$
curves over $\FF_8$ to see whether there is a curve with real Weil polynomial 
equal to $h$ (and indeed Savitt~\cite{SavittLauter2003} successfully implemented
a search through a subset of these curves), but for larger fields enumeration
would be out of reach.
\end{example}

The proof of part~(c) of Theorem~\ref{T:PPAV} relies on a theorem of 
Deligne~\cite{Deligne1969} that provides an equivalence between the category of
ordinary abelian varieties over a finite field and the category of 
\emph{Deligne modules}, which are finitely generated free $\ZZ$-modules 
provided with an action of Frobenius and Verschiebung~\cite[\S 4]{Howe1995}. 
Centeleghe and Stix~\cite{CentelegheStix2015} have given a similar equivalence
of categories for certain not-necessarily ordinary abelian varieties over finite
prime fields, and Bergström, Karemaker, and 
Marseglia~\cite{BergstromKaremakerEtAl2021} have provided a translation of the 
concept of a polarization under this equivalence. It is therefore possible that
a version of Theorem~\ref{T:PPAV}(c) (or its generalization to nonsimple 
ordinary isogeny classes~\cite[Theorem~1.1, p.~2362]{Howe1995}) could be proven
for these varieties.

Theorem~\ref{T:PPAV}, combined with work of Maisner and 
Nart~\cite{MaisnerNartEtAl2002} and an analysis of simple supersingular abelian 
surfaces~\cite{HoweMaisnerEtAl2008}, results in a complete classification of the
isogeny classes of abelian surfaces over finite fields that do not contain 
principally polarized varieties.

\begin{theorem}
\label{T:surfaces}
Let $x^4 + ax^3 + bx^2 + aqx + q^2$ be the Weil polynomial of an isogeny class 
of abelian surfaces over~$\Fq$. Then the isogeny class does not contain a
principally polarized surface if and only if $b$ is negative, $a^2 - b = q$, and
all prime divisors of $b$ are congruent to $1$ modulo~$3$.
\end{theorem}

\begin{proof}
This is \cite[Theorem 1, p.~121]{HoweMaisnerEtAl2008}.
\end{proof}

Note that no power of $3$ can be written in the form $a^2 - b$ for a negative 
integer $b$ with all prime divisors congruent to $1$ modulo~$3$, so every 
abelian surface over a finite field of characteristic $3$ is isogenous to a 
principally polarized variety. For a power $q$ of a prime $p\ne 3$, we can at 
least say the following: There are approximately 
$(\frac{32}{3})(1-\frac{1}{p})q^{3/2}$ isogeny classes of abelian surfaces over
$\Fq$~\cite[Theorem~1.1, p.~497]{DiPippoHowe1998}, and there are at most
$2\sqrt{q}$ pairs $(a,b)$ that satisfy the conditions of 
Theorem~\ref{T:surfaces}. In fact, Friedlander and 
Shparlinski~\cite[p.~2616]{FriedlanderShparlinski2014} show that under the 
assumption of the Riemann hypothesis for certain quadratic characters, the rate 
of growth of the number of such pairs is $O((\sqrt{q/\log q})\log \log q)$ as
$q\to\infty$. Thus, isogeny classes of abelian surfaces over $\Fq$ not 
containing principally polarized surfaces are relatively rare, and become more
so as~${q\to\infty}$.

%%%%%%%%%%%%%%%%%%%%%%%%%%%%%%%%%%%%%%%%%%%%%%%%%%%%%%%%%%%%%%%%%%%%%%%%%%%%%%%%
\subsection{Indecomposable principally polarized varieties}
\label{ssec:noindecomposable}

Suppose $A$ is an abelian variety over a field $K$ with a principal 
polarization~$\lambda$. We say that the polarized variety $(A,\lambda)$ is 
\emph{decomposable} if there exist principally polarized varieties 
$(A_1,\lambda_1)$ and $(A_2,\lambda_2)$ of positive dimension and an isomorphism
$A\to A_1\times A_2$ that identifies $\lambda$ with $\lambda_1\times\lambda_2$;
otherwise we say that $(A,\lambda)$ is \emph{indecomposable}. 

If $L$ is an extension field of $K$ we can consider the base extension 
$(A_L, \lambda_L)$ of the polarized variety $(A,\lambda)$ to~$L$. It is entirely
possible for $(A,\lambda)$ to be indecomposable while $(A_L, \lambda_L)$ is 
not. If $(A_L, \lambda_L)$ is indecomposable for an algebraic closure $L$ 
of~$K$, we say that $(A,\lambda)$  is \emph{geometrically indecomposable}.

Suppose $C$ is a curve over $K$ of genus~$g>0$, with canonically polarized 
Jacobian $(J,\lambda)$. As Serre notes~\cite[\S II.4]{Serre2020}, the polarized 
variety $(J,\lambda)$ is geometrically indecomposable: If it were decomposable 
then the theta-divisor on $J$ would be reducible, but the theta-divisor is the
image in $J$ of the $(g-1)$\up{th} symmetric power of~$C$, which is irreducible. 

Sometimes one can show that every principally polarized variety in an isogeny 
class is decomposable (possibly over an extension field), and in this situation
it follows immediately that the isogeny class does not contain a Jacobian. Each
of the following four sections gives a method that can sometimes be used to 
accomplish this.

%%%%%%%%%%%%%%%%%%%%%%%%%%%%%%%%%%%%%%%%%%%%%%%%%%%%%%%%%%%%%%%%%%%%%%%%%%%%%%%%
\subsection{The resultant 1 method, and splitting abelian varieties}
\label{ssec:resultant1}

In Theorem~2.4.1 of his course notes~\cite{Serre2020}, Serre gives a condition
that implies that an isogeny class of abelian varieties over $\Fq$ contains no
indecomposable principally polarized varieties. This has been referred to as the
``resultant~1" method \cite{Howe2021, HoweLauter2003}, and we will continue to 
use this terminology here. 

\begin{theorem}
\label{T:resultant1}
Suppose the real Weil polynomial $h$ of an isogeny class of abelian varieties
over $\Fq$ can be written as a product $h = h_1 h_2$ in~$\ZZ[x]$, where the
resultant of $h_1$ and $h_2$ is~$\pm 1$. Then every principally polarized 
variety in the isogeny class is decomposable, and the isogeny class contains no 
Jacobians.
\end{theorem}

\begin{proof} 
This is a restatement of Theorem 2.4.1 (p.~16) in Serre's 
notes~\cite{Serre2020}, making use of the remark following the statement of the
theorem.
\end{proof}

The resultant~1 method fits into a more general context, which we provide here
for perspective and in preparation for later results. We begin with a result
about certain split isogeny classes.

\begin{proposition}
\label{P:splitting}
Let $\calI_1$ and $\calI_2$ be isogeny classes of abelian varieties over $\Fq$
with real Weil polynomials $h_1$ and~$h_2$, respectively, and suppose $h_1$ and 
$h_2$ are coprime in~$\QQ[x]$. Let $\calI$ be the isogeny class that contains 
all varieties isogenous to $A_1\times A_2$ for $A_i\in \calI_i$, so that the 
real Weil polynomial of $\calI$ is $h_1 h_2$.

Then for every $A\in\calI$, there are unique elements $A_1\in\calI_1$ and 
$A_2\in\calI_2$ and a unique finite group scheme $\Delta$ such that there is an
exact sequence
\begin{equation}
\label{EQ:delta}
0 \longrightarrow \Delta \longrightarrow A_1\times A_2\longrightarrow A \longrightarrow 0
\end{equation}
with the property that the natural projections $A_1\times A_2 \to A_1$ and 
$A_1\times A_2\to A_2$ give embeddings $\Delta\hookrightarrow A_1$ and 
$\Delta\hookrightarrow A_2$.

If $A$ has a principal polarization $\lambda$, then $\lambda$ pulls back to a
product polarization $\lambda_1\times \lambda_2$ on $A_1\times A_2$, and the
embeddings $\Delta\hookrightarrow A_i$ obtained from the projection maps give
isomorphisms between $\Delta$ and $\ker\lambda_i$, for each~$i$. In particular, 
if $A$ has a principal polarization then $\Delta$ is self-dual.
\end{proposition}

\begin{proof}
Let $B_1$ and $B_2$ be elements of $\calI_1$ and $\calI_2$, respectively. Since 
$h_1$ and $h_2$ are coprime, we have $\Hom(B_1,B_2)=\Hom(B_2,B_1)= 0$, so 
$\End (B_1\times B_2)\cong \End B_1\times \End B_2$. For each $i$, let $E_i$
denote $(\End B_i)\otimes \QQ$.

Let $\psi$ be an isogeny from $A$ to $B_1\times B_2$. Then $\psi$ induces an 
isomorphism $(\End A)\otimes \QQ\to E_1\times E_2$. Let $n$ be a positive
integer such that the element $(n,0)$ of $E_1\times E_2$ is in the image of 
$\End A$ under this isomorphism. Let $\alpha\in \End A$ be the element that maps
to $(n,0)\in E_1 \times E_2$, and let $A_1$ be the image of $A$ under the 
endomorphism~$\alpha$. Then $A_1$ is a sub-abelian variety of $A$ that has real
Weil polynomial~$h_1$, and indeed $A_1$  is characterized by this property.
Define $A_2$ likewise. The intersection $A_1\cap A_2$ is finite, again because 
$h_1$ and $h_2$ are coprime.

The two homomorphisms $A_i\hookrightarrow A$ give us an isogeny 
$\varphi\colon A_1\times A_2\to A$. Let $\Delta$ be its kernel. If the 
composition $\Delta\to A_1\times A_2 \to A_1$ had a nontrivial kernel, then so
would the homomorphism $A_1\hookrightarrow A$, which is an embedding. Therefore
the maps $\Delta\to A_1$ and $\Delta\to A_2$ identify $\Delta$ with closed
subgroup schemes of the~$A_i$.

Suppose $A_1'$, $A_2'$, and $\Delta'$ also have the properties described in the
theorem. Then the composition $A_1'\to A_1'\times A_2'\to A$ has trivial kernel
and identifies $A_1'$ as an abelian subvariety of $A$ with real Weil 
polynomial~$h_1$, so $A_1'\cong A_1$, and likewise $A_2'\cong A_2$. Under these
isomorphisms, the isogeny $A_1'\times A_2'\to A$ is identified with 
$\varphi\colon A_1\times A_2\to A$, so we also have $\Delta'\cong \Delta$. This
proves the uniqueness of $A_1$, $A_2$, and~$\Delta$.

Suppose $\lambda$ is a principal polarization of~$A$. Since $\Hom(A_1,A_2)$ and
$\Hom(A_2,A_1)$ are both trivial, the pullback of $\lambda$ via the map 
$\varphi\colon A_1\times A_2\to A$ must be a product polarization on 
$A_1\times A_2$, say $\varphi^*\lambda = \lambda_1\times\lambda_2$. The degree
of $\varphi^*\lambda$ is the square of the degree of $\varphi$, so we have 
$(\#\Delta)^2 = \deg(\lambda_1\times \lambda_2) = (\deg \lambda_1)(\deg\lambda_2)$,
where we use $\#$ to indicate the rank of a finite group scheme. Since $\Delta$ 
is a subgroup scheme of $\ker\lambda_1\times\ker\lambda_2$ and since each 
projection $A_1\times A_2\to A_i$ gives an embedding 
$\Delta\hookrightarrow A_i$, we see that $\Delta\hookrightarrow\ker\lambda_i$ as
well, so $\#\Delta\le\deg\lambda_i$ for each~$i$. Therefore 
$\#\Delta = \deg\lambda_i$ for each~$i$, so each 
$\Delta\hookrightarrow \ker\lambda_i$ is an isomorphism.
\end{proof}

Next we connect the decomposition described in Proposition~\ref{P:splitting}
with resultants. 

\begin{definition}[{\cite[p.~179]{Pohst1991}}]
The \emph{reduced resultant} of two polynomials $f,g \in \ZZ[x]$ is the
non-negative generator of the ideal $(f,g)\cap \ZZ$ of~$\ZZ$. 
(Note that the reduced resultant divides the usual resultant, and is divisible
by every prime divisor of the usual resultant.)
\end{definition}

\begin{definition}
Let $K$ be a field. The \emph{radical} of a nonzero element $f$ of the 
polynomial ring $K[x]$ is the monic squarefree polynomial of largest degree that
divides~$f$. Equivalently, the radical of $f$ is the product of the monic 
irreducible factors of~$f$, each taken once.
\end{definition}

\begin{proposition}
\label{P:delta}
Let $\calI$ be an isogeny class of abelian varieties over~$\Fq$ and let $h$ be 
its real Weil polynomial. Suppose $h$ can be factored in $\ZZ[x]$ as $h_1 h_2$,
where $h_1$ and $h_2$ are monic polynomials with no common factor. Then each
$h_i$ is the real Weil polynomial of an isogeny class $\calI_i$ over~$\Fq$.

Let $g_1$ and $g_2$ be the radicals of $h_1$ and~$h_2$ and let $r$ be the
reduced resultant of $g_1$ and~$g_2$. Then $r\neq 0$, and for every 
$A\in \calI$ the group scheme $\Delta$ that appears in the 
decomposition~\eqref{EQ:delta} provided by
Proposition~\textup{\ref{P:splitting}} is annihilated by~$r$.
\end{proposition}

\begin{proof}
This follows from results of Howe and Lauter~\cite[\S 2]{HoweLauter2012}, but we
provide a self-contained proof here.

The Honda--Tate theorem~\cite[Théorème~1, p.~96]{Tate1971} shows that each
$h_i$ is the real Weil polynomial of its own isogeny class $\calI_i$ of abelian
varieties over~$\Fq$.

Since $h_1$ and $h_2$ have no common factors, neither do $g_1$ and~$g_2$, and
this is enough to show that $r\neq 0$.

Let $F$ and $V$ denote the Frobenius and Verschiebung endomorphisms of~$\Delta$,
and let $\pi_i$ and $\pibar_i$ denote the Frobenius and Verschiebung
endomorphisms of~$A_i$. The actions of $F$ and $V$ on $\Delta$ are consistent 
(under the embeddings $\Delta\hookrightarrow A_i$) with the actions of $\pi_i$
and $\pibar_i$ on $A_1$ and~$A_2$, so $F+V$ satisfies every polynomial in 
$\ZZ[x]$ that is satisfied by $\pi_1 + \pibar_1$ or by $\pi_2 + \pibar_2$. In
particular, $F+V$ satisfies $h_1$ and~$h_2$, but we can say more. The 
Honda--Tate theorem shows that $\pi_i$ and $\pibar_i$ lie in the center $Z$ of 
$\End A_i$, and $Z\otimes \QQ$ is a product of number fields, so the minimal 
polynomials of $\pi_1 + \pibar_1$ and $\pi_2 + \pibar_2$ are both squarefree.
Therefore for each $i$ we see that $\pi_i + \pibar_i$ satisfies~$g_i$. Since 
$F+V$ satisfies both $g_1$ and~$g_2$, it also satisfies every $\ZZ[x]$-linear
combination of the two. The reduced resultant of $g_1$ and $g_2$ is a 
$\ZZ[x]$-linear combination of $g_1$ and~$g_2$, so the finite group scheme
$\Delta$ is annihilated by the integer~$r$. 
\end{proof}

Serre's resultant~1 method fits into this framework. The hypothesis in 
Theorem~\ref{T:resultant1} that the resultant of $h_1$ and $h_2$ is $\pm1$ shows
that the reduced resultant of the radicals of $h_1$ and $h_2$ is~$1$, so every
$A$ in the isogeny class is \emph{isomorphic} to a product $A_1\times A_2$ of 
abelian varieties with no isogeny factors in common. The only polarizations of
such varieties are product polarizations, so in particular every principle
polarization of $A$ is decomposable.

\begin{remark}
Let $A$, $A_1$, and $A_2$ be as in Proposition~\ref{P:splitting}, and let $\pi$,
$\pi_1$, and $\pi_2$ be the Frobenius endomorphisms of these three abelian 
varieties. Then we have an embedding
\[
\ZZ[\pi,\pibar]\hookrightarrow \QQ(\pi)\cong \QQ(\pi_1)\times \QQ(\pi_2).
\]
Let $r$ be the smallest positive integer such that 
$(0,r)\in \QQ(\pi_1)\times \QQ(\pi_2)$ lies in the image of $\ZZ[\pi,\pibar]$. 
Then the proof of Proposition~\ref{P:delta} goes through with this $r$ in place
of the reduced resultant of the radicals of the~$h_i$. Usually this $r$ is equal
to the reduced resultant, but there are circumstances in which it is half that;
see~\cite[Proposition~2.8, p.~178]{HoweLauter2012}.
\end{remark}

%%%%%%%%%%%%%%%%%%%%%%%%%%%%%%%%%%%%%%%%%%%%%%%%%%%%%%%%%%%%%%%%%%%%%%%%%%%%%%%%
\subsection{Hermitian modules}
\label{ssec:hermitian}

In his 1985 Harvard course, Serre used the concept of \emph{Hermitian modules}
to study isogeny classes of powers of elliptic curves. He fleshed out the theory
in a letter to Lauter~\cite[Appendix]{LauterSerre2002}), as well as in his
course notes~\cite[\S 3.8]{Serre2020}. As is noted in the latter reference, a 
much fuller exposition in a more general setting is given by Jordan, Keeton, 
Poonen, Rains, Shepherd-Barron, and Tate~\cite{JordanKeetonEtAl2018}.

What is relevant for our discussion in this section is the characterization of
the isogeny classes that contain powers of ordinary elliptic curves and that do 
not contain any indecomposable principally polarized varieties. Combining the 
theory of Hermitian modules with results of Hoffmann, we find the following.

\begin{theorem}
\label{T:hoffmann}
Let $E$ be an elliptic curve over a finite field $\Fq$ with Weil polynomial
$x^2 - tx + q$, where $t^2 \ne 4q$. 
\begin{itemize}
\item[(a)] There is an abelian surface isogenous to $E^2$ that has an 
           indecomposable principal polarization if and only if $t^2 - 4q$ is 
           not in $\{-3,-4,-7\}.$
\item[(b)] There is an abelian variety isogenous to $E^3$ that has an 
           indecomposable principal polarization if and only if $t^2 - 4q$ is
           not in $\{-3,-4,-8,-11\}.$
\end{itemize}
\end{theorem}

\begin{proof}
Using the theory of Hermitian modules, each statement reduces to the question of
the existence or nonexistence of a Hermitian form over an $R$-lattice of the 
appropriate rank, where $R$ is the imaginary quadratic order of discriminant 
$t^2 - 4q$. Both statements then follow from results of
Hoffmann~\cite[Theorems~8.1 and~8.2, p.~424]{Hoffmann1991}.
\end{proof}

We see that with Hermitian modules we can indeed show that there are a few 
isogeny classes of powers of elliptic curves that do not contain abelian 
varieties with an indecomposable principal polarization. When we consider powers
of large enough dimension, however, this method of producing such examples no 
longer works, as the following theorem shows.

\begin{theorem}
\label{T:smith}
Let $E$ be an elliptic curve over a finite field $\Fq$ with Weil polynomial
$x^2 - tx + q$, where $t^2 \ne 4q$. If $n = 8$, $n = 12$, or $n > 13$, then
there is an abelian variety isogenous to $E^n$ that has an indecomposable
principal polarization.
\end{theorem}

\begin{proof}
Let $R$ be the maximal order of the field $\QQ(\sqrt{t^2 - 4q}).$ 
O'Meara~\cite[\S 6]{OMeara1975}, building on work of Erd\H{o}s and 
Ko~\cite[Theorem~1, p.~103]{ErdosKo1939}, shows that for the values of $n$
mentioned in the statement of the theorem there exists an indecomposable
unimodular $\ZZ$-lattice of rank~$n$. Smith~\cite[Theorem~2, p.~1025]{Smith1978}
shows that each such form remains indecomposable when it is tensored with $R$ to
produce a Hermitian form on the $R$-lattice~$R^n$. These forms provide us with
indecomposable principal polarizations on $n$-dimensional abelian varieties
isogenous to~$E^n$.
\end{proof}

The author is not aware of any analogs of Theorem~\ref{T:hoffmann} for the
values of $n>3$ not covered by Theorem~\ref{T:smith}, although results have been
obtained for some specific quadratic orders. For example, for the quadratic 
order $R$ of discriminant~$-3$, results of Feit~\cite{Feit1978} and 
Abdukhalikov~\cite{Abdukhalikov2004}, combined with Theorem~\ref{T:smith}, show 
that there are indecomposable unimodular Hermitian $R$-lattices of every rank 
other than $1$, $2$, $3$, $4$, $5$, and~$7$.

%%%%%%%%%%%%%%%%%%%%%%%%%%%%%%%%%%%%%%%%%%%%%%%%%%%%%%%%%%%%%%%%%%%%%%%%%%%%%%%%
\subsection{Counting principally polarized varieties}
\label{ssec:counting}

If $\calI$ is an isogeny class of ordinary abelian varieties, it is possible in
principle to use Deligne's equivalence of categories~\cite{Deligne1969} and the
analysis of polarizations in the category of Deligne modules~\cite{Howe1995} to
enumerate the principally polarized varieties in~$\calI$. For those ordinary 
isogeny classes whose real Weil polynomials are either squarefree or a power of
a single irreducible polynomial, Marseglia~\cite{Marseglia2019,Marseglia2021} 
provides explicit algorithms for doing so. For some simple ordinary isogeny 
classes, there are even formulas for the number of principally polarized 
varieties in terms of ratios of certain class numbers. In this section, we look
at a family of isogeny classes for which there is a class number formula for the
number of principally polarized varieties, and a different class number formula
for the number of geometrically decomposable principally polarized varieties.
For the family we have in mind, the Brauer class number relations show that the 
two formulas give identical results, so there are no geometrically 
indecomposable principally polarized varieties in these isogeny classes. We 
begin by reviewing some results about counting principally polarized varieties
in ordinary isogeny classes.

Let $\calI$ be a simple ordinary isogeny class of abelian varieties over~$\Fq$,
so that its Weil polynomial $f$ defines a CM~field~$K$. Let $\pi$ be a root of
$f$ in $K$ and let $\pibar$ be its complex conjugate. Then the order
$\ZZ[\pi,\pibar]$ can be embedded into the endomorphism ring of every variety in
$\calI$ by sending $\pi$ to Frobenius and $\pibar$ to Verschiebung. Let $\calO$ 
be the maximal order of~$K$. 
Waterhouse~\cite[Theorem~7.4, p.~554]{Waterhouse1969} shows that for every order
$R$ with $\ZZ[\pi,\pibar]\subseteq R\subseteq \calO$, there are abelian 
varieties in $\calI$ with endomorphism ring isomorphic to $R$ as a 
$\ZZ[\pi,\pibar]$-algebra. If $A$ is a variety in $\calI$ and $\Ahat$ its dual,
then their endomorphism rings (viewed as subrings of $\calO$) are taken to one 
another by complex conjugation. Since a principally polarized variety is 
isomorphic to its dual, the endomorphism rings of principally polarized 
varieties are stable under complex conjugation.

Let $\calOplus$ be the maximal real subring of~$\calO$, let $U$ be the unit
group of~$\calO$, and let $\Uplus$ be the unit group of $\calOplus$. Results of 
Shimura and Taniyama (\cite[{\S}14]{ShimuraTaniyama1961}, 
\cite[{\S}14]{Shimura1998}) show that the number of principally polarized 
varieties over $\CC$ (of a given CM-type) with endomorphism ring $\calO$ is 
equal to 
\[
\frac{1}{[N(U) : (\Uplus)^2]} \, \frac{\#\Pic \calO^{\phantom{+}}}{\#\Pic \calOplus},
\]
where $N$ is the norm from $\calO$ to $\calOplus$ and where $\Pic$ is the Picard
group functor. Using Deligne modules, one finds that the same expression counts 
the number of principally polarized varieties in our isogeny class $\calI$ whose
endomorphism rings are isomorphic to~$\calO$.

There are similar class number formulas for the number of principally polarized
varieties in $\calI$ with other endomorphism rings; see for example
\cite[{\S}8]{LenstraPilaEtAl2002}, \cite[Proposition~2, p.~583]{Howe2004},
and~\cite[Lemma~19, p.~398]{IonicaThome2020}. The most general class of rings 
for which the author knows this type of formula to hold is the class of 
\emph{convenient} orders~\cite[\S2]{Howe2020}; an order 
$R\supseteq \ZZ[\pi,\pibar]$ is convenient if it is stable under complex
conjugation, its maximal real suborder is Gorenstein, and its trace dual is
generated as an $R$-module by its totally imaginary elements.

These counting results lead to a surprising family of examples, which was
initially discovered experimentally by Maisner and 
Nart~\cite[\S 4.1]{MaisnerNartEtAl2002} in the course of gathering statistics on
the Weil polynomials of genus-$2$ curves over small finite fields.

\begin{theorem}
\label{T:MN2}
Let $q$ be an odd prime power and let $\calI_q$ be the isogeny class of abelian
surfaces over $\Fq$ with Weil polynomial $x^4 + (2-2q)x^2 + q^2$. Then there are
no geometrically indecomposable principally polarized abelian surfaces
in~$\calI_q$.
\end{theorem}

\begin{proof}
A geometrically indecomposable principally polarized abelian surface is exactly
the Jacobian of a genus-$2$ curve, so this theorem follows 
from~\cite[Theorem~1, p.~581]{Howe2004}. We refer the reader to that paper for 
the complete proof; here we simply present an outline to show how the proof
connects to the counting formulas discussed above.

Write $2q - 1 = F^2 D$ for $D$ squarefree. Then the number field 
$K\colonequals \QQ(\pi)$, with $\pi$ satisfying $x^4 + (2-2q)x^2 + q^2$, is 
isomorphic to $\QQ(\sqrt{-2},\sqrt{-D}).$ If $\calO$ is the maximal order 
of~$K$, then for every divisor $f$ of $F$ we have an order 
$R_f\colonequals \ZZ + fw\calO$, where $w = (\sqrt{2D} + \sqrt{-2})/2\in K$.
We check that the orders $R_f$ are the only orders of $\calO$ that are stable 
under complex conjugation and that contain~$\pi$, and they are all convenient. 
It follows from the discussion above that the number of principally polarized
varieties in $\calI_q$ is equal to
\[
\sum_{f\mid F} 
\frac{1}{[N(U_f) : (U_f^{+})^2]} \, \frac{\#\Pic R_f^{\phantom{+}}}{\#\Pic R_f^{+}},
\]
where $U_f$ is the unit group of $R_f$ and $U_f^{+}$ is the unit group 
of~$R_f^{+}$.

The geometrically decomposable principally polarized surfaces in $\calI_q$ can
all be constructed as follows. Let $\calJ_q$ denote the isogeny class of
elliptic curves over $\Fqtwo$ with Weil polynomial $x^2 + (2-2q)x + q^2$. Given
an elliptic curve $E\in \calJ_q$, let $A$ be the restriction of scalars of $E$
from $\Fqtwo$ to~$\Fq$. The principal polarization on $E$ gives a principal 
polarization $\lambda$ on~$A$, and $(A,\lambda)$ decomposes over $\Fqtwo$ as the
product of the polarized variety $E$ with its conjugate $E^{(q)}$ over~$\Fq$. 
Two elliptic curves $E_1$ and $E_2$ in $\calJ_q$ give rise to isomorphic
polarized surfaces $(A,\lambda)$ if and only if either $E_2\cong E_1$ or 
$E_2\cong E_1^{(q)}$. So for most~$q$, the number of decomposable principally
polarized surfaces is simply $\#\calJ_q / 2$; the exceptions are those $q$ for
which $\calJ_q$ contains elliptic curves $E$ with $E\cong E^{(q)}$. When $D = 1$
there is one such~$E$, and otherwise there are no such~$E$. Let us consider the
case~$D>1$. 

If we let $L$ be the subfield $\QQ(\sqrt{-D})$ of~$K$, then the possible 
endomorphism rings of elliptic curves in $\calJ_q$ are the orders
$S_f \colonequals \ZZ[f\sqrt{-D}]$ for the divisors $f$ of~$F$. We find that
\[
\#\calJ_q = \sum_{f\mid F} \#\Pic S_f.
\]
Now, since the class number of $\QQ(\sqrt{-2})$ is~$1$, the Brauer 
relations~\cite[\S VIII.7]{FrohlichTaylor1993} applied to the quartic 
$V_4$-extension $K/Q$ show that 
\[\frac{\Reg \Kplus}{\Reg K^{\phantom{+}}} \,
  \frac{\#\Pic R_1^{\phantom{+}}}{\#\Pic R_1^{+}} 
  = \frac{1}{4}\, \#\Pic S_1,
\]
where $\Reg K$ and $\Reg \Kplus$ are the regulators of $K$ and~$\Kplus$. We
calculate that the ratio $(\Reg K) / (\Reg \Kplus)$ is equal to twice the index
$[N(U_1) : (\Uplus_1)^2]$, so we obtain
\[\frac{1}{[N(U_1) : (\Uplus_1)^2]} \,
  \frac{\#\Pic R_1^{\phantom{+}}}{\#\Pic R_1^{+}} 
  = \frac{1}{2}\, \#\Pic S_1.
\]
We can build upon this relation to find that 
\[\frac{1}{[N(U_f) : (\Uplus_f)^2]} \,
  \frac{\#\Pic R_f^{\phantom{+}}}{\#\Pic R_f^{+}} 
  = \frac{1}{2}\, \#\Pic S_f
\]
for every divisor $f$ of~$F$, so by the counting formulas we discussed above, 
the number of principally polarized abelian surfaces in $\calI_q$ is equal to 
the number of decomposable principally polarized abelian surfaces in $\calI_q$.
A similar calculation with some added complications gives us the same result
when~$D = 1$.
\end{proof}

The same technique can be used to prove that for all~$q$, the isogeny class of 
abelian surfaces with Weil polynomial $x^4 + (1 - 2q)x^2 + q^2$ does not contain
any geometrically indecomposable principally polarized varieties, but as we will
see in Section~\ref{ssec:nonsplit}, there is a much simpler way of proving this 
result.

On the other hand, Maisner~\cite{Maisner2004} uses these same counting 
techniques to show that every ordinary isogeny class whose Weil polynomial is of
the form ${x^4 + ax^2 + q^2}$ with ${a\ne 1-2q}$ and ${a\ne 2-2q}$ \emph{does}
contain a geometrically indecomposable principally polarized variety.

It would be interesting to find other examples of isogeny classes of simple
ordinary abelian varieties over finite fields that contain principally polarized
varieties but that do not contain geometrically indecomposable principally
polarized varieties. The author suspects that such examples might be rare.

%%%%%%%%%%%%%%%%%%%%%%%%%%%%%%%%%%%%%%%%%%%%%%%%%%%%%%%%%%%%%%%%%%%%%%%%%%%%%%%%
\subsection{Supersingular factors}
\label{ssec:supersingular}

For abelian varieties that are isogenous to a product of an ordinary abelian
variety with a power of a supersingular elliptic curve with all endomorphisms
rational, we have one more technique for showing that all principal
polarizations are decomposable.

\begin{theorem}
\label{T:supersingular}
Suppose $q$ is a square prime power, and let $s$ be one of the two integer 
square roots of~$q$. If the real Weil polynomial of an isogeny class can be
written as the product of a nonconstant ordinary real Weil polynomial 
$h_0\in\ZZ[x]$ and $(x-2s)^n$, for some~$n>0$, and if the integer $h_0(2s)$ is
squarefree, then every principally polarized variety in the isogeny class is 
decomposable.
\end{theorem}

\begin{proof}
This is essentially~\cite[Theorem~3.1, p.~180]{HoweLauter2012}; the conclusion
of that theorem is that there are no Jacobians in the isogeny class, but this is
proven by showing that there are no indecomposable principally polarized 
varieties in the isogeny class. Here we will only consider the special case when
$h_0$ is irreducible, in the hope of communicating the spirit of the full proof 
without getting bogged down by details.

The main tool we use is the combination of Propositions~\ref{P:splitting} 
and~\ref{P:delta}. The idea is that these results give us finite group schemes
$\Delta$ of bounded exponent that can be embedded into two different abelian
varieties $A_1$ and~$A_2$, so that the Frobenius on $\Delta$ has to behave 
according to the Weil polynomials of both $A_1$ and $A_2$ simultaneously.
Frobenius acts as a rational integer on varieties with real Weil polynomial 
$(x - 2s)^n$, and that greatly limits the possibilities for~$\Delta$.

Suppose, to obtain a contradiction, that we have an indecomposable principally
polarized variety $(A,\lambda)$ in the isogeny class from the theorem. If we 
take $h_1 = h_0$ and $h_2 = (x - 2s)^n$, Proposition~\ref{P:splitting} provides
us with an ordinary variety~$A_1$, a supersingular variety $A_2 \cong E^n$, and
a finite group scheme $\Delta$ that can be embedded in both $A_1$ and~$A_2$. Let
$F$ and $V$ be the Frobenius and Verschiebung on~$A_1$, and let $R$ be the 
subring $\ZZ[F,V]$ of $\End A_1$. By the theory of Deligne 
modules~\cite[\S 4]{Howe1995}, the finite group schemes that can be embedded in
$A_1$ can be understood in terms of certain $R$-modules of finite cardinality.
For a group scheme $G\hookrightarrow A_1$ of rank coprime to~$q$, this 
association simply assigns to $G$ the $R$-module that is isomorphic as an 
abelian group to the group of geometric points of~$G$, and whose $R$-module
structure is given by having $F$ act as Frobenius.

The $\Delta$ associated to our principally polarized variety is nontrivial
because we assumed our polarized variety is indecomposable. 
Proposition~\ref{P:delta} shows that $\Delta$ is killed by the squarefree
integer $h_0(2s)$, and we check that $h_0(2s)$ is coprime to~$q$, because 
$s^2 = q$ and the constant term of $h_0$ is coprime to $q$ since $h_0$ is 
ordinary. Since $\Delta$ is nontrivial, we can choose a prime $\ell$ that 
divides~$\#\Delta$, and we know that $\ell$ is coprime to~$q$. Let $G$ be the 
$\ell$-power torsion part of~$\Delta$. Again because $\Delta$ is killed by a 
squarefree integer, we see that $G = \Delta[\ell]$.

Let $M$ be the finite $R$-module associated to~$G$. Then $R$ acts on $M$ via 
reduction to the ring $R/\frakp$, where $\frakp$ is the ideal $(\ell,F-s,V-s)$
of~$R$. It is clear that $R/\frakp\cong\FF_\ell$, so the ideal $\frakp$ is
prime.

Now we again use the fact that $h_0(2s)$ is squarefree. Since $\frakp$ contains
both $F-s$ and~$V-s$, its square $\frakp^2$ contains 
$(F-s)(V -s) = s(2s - F - V)$, and since $s$ is coprime to $\ell$ and hence not
in~$\frakp$, we see that $F + V \equiv 2s\bmod \frakp^2$. Therefore
$0 = h_0(F + V) \equiv h_0(2s)\bmod \frakp^2$. Since $h_0(2s)$ is squarefree,
one of its prime divisors must lie in $\frakp^2$, and since $\ell\in\frakp$, 
this prime must be~$\ell$.

Using the congruences $V \equiv 2s - F\bmod \frakp^2$ and 
$\ell \equiv 0\bmod \frakp^2$, it is straightforward to show that every element
of $R$ is congruent modulo $\frakp^2$ to an element of the form $aF + b$, with
$0\le a,b < \ell$. It follows that $\#(R/\frakp^2) = \ell^2$ and that 
$\frakp/\frakp^2$ is a $1$-dimensional $R/\frakp$-module, so the localization 
$R_\frakp$ is a discrete valuation ring and the prime $\frakp$ is nonsingular.

Let $\frakA$ be the Deligne module associated with the ordinary variety~$A_1$, 
so that, by~\cite[Proposition~4.14, p.~2373]{Howe1995}, the $R$-module
associated to $A_1[\ell]$ is $\frakA / \ell\frakA$. Since $h_0$ is irreducible,
$\frakA$ can be viewed as a sub-$R$-module of $R\otimes\QQ$, and since $\frakp$
is a nonsingular prime of~$R$, the $\frakp$-primary part of $\frakA/\ell\frakA$
is isomorphic to the $\frakp$-primary part of $R/\ell R$. The latter is simply 
$R/\frakp^e$, where $e$ is the $\frakp$-adic valuation of~$\ell$.

Since the group scheme $\Delta$ can be embedded in $A_1$, the group scheme
$G = \Delta[\ell]$ can be embedded in $A_1[\ell]$, so the $R$-module $M$
associated to $G$ can be embedded in $R/\frakp^e$. Since $M$ is a 
$\frakp$-torsion module, $M$ can be embedded in the $\frakp$-torsion submodule
of $R/\frakp^e$, and since $\frakp$ is nonsingular, this torsion submodule is 
$\frakp^{e-1}/\frakp^e\cong R/\frakp$. Since $M$ is nontrivial, we must have
$M\cong R/\frakp$, so $\#M = \ell$. This shows that the rank of the group scheme
$G = \Delta[\ell]$ is~$\ell$. But $G$ is the $\ell$-power torsion part
of~$\Delta$, and so must have square rank, because polarizations have square 
degree and $\Delta$ is the kernel of a polarization, by 
Proposition~\ref{P:splitting}. This contradiction shows that there is no 
principally polarized abelian variety  in the isogeny class.
\end{proof}

%%%%%%%%%%%%%%%%%%%%%%%%%%%%%%%%%%%%%%%%%%%%%%%%%%%%%%%%%%%%%%%%%%%%%%%%%%%%%%%%
%%%%%%%%%%%%%%%%%%%%%%%%%%%%%%%%%%%%%%%%%%%%%%%%%%%%%%%%%%%%%%%%%%%%%%%%%%%%%%%%
\section{Deducing and using information about Jacobians}
\label{sec:deducing}

Suppose we have an isogeny class $\calI$ of abelian varieties over a finite
field~$\Fq$, and suppose we have been unable to prove that it contains no 
geometrically indecomposable principally polarized varieties. Suppose further
that the Weil polynomial $f$ of $\calI$ could conceivably be the Weil polynomial
of the Jacobian of a curve~$C$, in the sense that for every~$d>0$ the Weil
polynomial would not predict a negative number of degree-$d$ places on~$C$. 

For some such isogeny classes we can use the Weil polynomial to prove that every
Jacobian in the isogeny class must come from a curve with some special geometric
property. Sometimes the properties we can deduce are not consistent with other
facts we know about the curve. Other times, we can efficiently search through 
the set of curves with the given properties and check to see whether any of them
has Weil polynomial equal to~$f$.

In this section, we will describe three classes of properties of curves that we
can sometimes deduce from their Weil polynomials. First, there are Weil
polynomials that imply that a curve with that Weil polynomial must have a 
nontrivial automorphism of a known degree. Second, there are Weil polynomials 
that imply that a curve with that Weil polynomial must have a map of known 
degree to an elliptic curve of known trace. And third, there are Weil 
polynomials that imply that a curve with that Weil polynomial must be definable
over a proper subfield of the base field. We discuss the first type in 
Sections~\ref{ssec:automorphisms} through~\ref{ssec:nonsplit}, the second in 
Section~\ref{ssec:knowndegree}, and the third in Section~\ref{ssec:descent}.

%%%%%%%%%%%%%%%%%%%%%%%%%%%%%%%%%%%%%%%%%%%%%%%%%%%%%%%%%%%%%%%%%%%%%%%%%%%%%%%%
\subsection{Curves with nontrivial automorphisms: Introduction}
\label{ssec:automorphisms}

In Section~\ref{ssec:resultant1} we proved Proposition~\ref{P:delta}, which, 
roughly speaking, gives us a bound on how far away an abelian variety in a
nonsimple isogeny class is from a product of two nonzero varieties with no
isogeny factors in common. Serre's resultant~1 method makes use of this idea;
if an abelian variety is \emph{isomorphic} to a product of two such varieties,
all of its polarizations are decomposable.

It is also sometimes possible to use Proposition~\ref{P:delta} to show that
every curve whose Jacobian lies in a particular isogeny class must have an 
automorphism of a particular order. Before explaining this, let us review the
definition of the Rosati involution, the definition of an automorphism of a
polarized abelian variety, and the connection between automorphisms of a curve 
and automorphisms of its polarized Jacobian.

Let $A$ be an abelian variety over a field~$k$ and let $\lambda\colon A\to\Ahat$
be a polarization (not necessarily principal) of~$A$. The
\emph{Rosati involution} on $(\End A)\otimes \QQ$ associated to the polarization
$\lambda$ is the involution that sends an endomorphism $\beta$ to the element 
$\beta^\dagger \colonequals \lambda^{-1}\betahat\lambda$ of 
$(\End A)\otimes \QQ$. Let $\Tr$ denote the trace from $(\End A)\otimes \QQ$ 
to~$\QQ$. The bilinear form on $(\End A)\otimes \QQ$ that sends a pair
$(\beta,\gamma)$ to $\Tr \beta\gamma^\dagger$ is positive 
definite~\cite[Theorem~17.3, p.~138]{Milne1986a}. It is perhaps unfortunate that
the standard notation for the Rosati involution does not make reference to the
polarization $\lambda$, because in general different polarizations can give rise
to different involutions. However, we note that the only positive definite
involution on a totally real number field is the identity, and the only positive
definite involution on a CM field is complex conjugation. Since the center of
$(\End A)\otimes \QQ$ is a product of number fields that are each either totally
real or CM, the Rosati involution restricted to the center of
$(\End A)\otimes \QQ$ \emph{is} independent of the polarization used to define
it.

An \emph{automorphism} of the polarized variety $(A,\lambda)$ is an automorphism
$\alpha$ of $A$ such that the diagram 
\[
\xymatrix@C=5em{
A\ar[r]^\alpha\ar[d]_(0.4)\lambda & A\ar[d]^(0.4)\lambda\\
\Ahat                        & \Ahat\ar[l]^{\alphahat}\\
}
\]
commutes. Equivalently, $\alpha$ is an automorphism of $(A,\lambda)$ if and only
if $\alpha^\dagger\alpha = 1$, where $\dagger$ denotes the Rosati involution
associated to $\lambda$. We see, for example, that if $(\End A)\otimes \QQ$ is a
CM field, then $\alpha\in\End A$ is an automorphism of $(A,\lambda)$ if and only
if $\alpha$ is a root of unity.

Finally, if $\varepsilon$ is an automorphism of a curve $C$ over a field~$k$,
then the pullback $\varepsilon^*$ gives an automorphism of the polarized
Jacobian $(J,\lambda)$ of~$C$, and we have the following result.

\begin{proposition}
\label{P:torelli}
The map $\varepsilon\mapsto(\varepsilon^{-1})^*$ gives a group homomorphism
$\varphi\colon\Aut C\to\Aut(J,\lambda)$ such that
\begin{itemize}
\item[(a)] $\varphi$ is injective\textup{;}
\item[(b)] if $C$ is hyperelliptic, then $\varphi$ is an isomorphism\textup{;}
\item[(c)] if $C$ is not hyperelliptic, then the cokernel of $\varphi$ 
           has order~$2$, and in fact 
           \[\Aut(J,\lambda) \cong \{\pm1\} \times \varphi(\Aut C).\]
\end{itemize}           
\end{proposition}

\begin{remark}
Note that if $\alpha$ and $\beta$ are automorphisms of $C$, then 
$(\alpha\beta)^* = \beta^*\alpha^*$; the perhaps-unexpected inverse that appears
in the formula for $\varphi$ is there so that we have
$\varphi(\alpha\beta) =\varphi(\alpha)\varphi(\beta)$.
\end{remark}

\begin{proof}[Proof of Proposition~\textup{\ref{P:torelli}}]
The proposition follows immediately from Torelli's 
theorem~\cite[Theorem~12.1, p.~202]{Milne1986b}.
\end{proof}

In particular, item (c) shows that if $\alpha$ is an automorphism of 
$(J,\lambda)$, then there is an automorphism $\varepsilon$ of $C$ such that 
either $\alpha = \varepsilon^*$ or $\alpha = -\varepsilon^*$.

With this background information in hand, we can describe how some Weil 
polynomials imply that every curve with that Weil polynomial has nontrivial
automorphisms.

\begin{proposition}
\label{P:aut}
Let $\calI_1$ and $\calI_2$ be isogeny classes of nonzero abelian varieties over
$\Fq$ with real Weil polynomials $h_1$ and~$h_2$, respectively, and suppose 
$h_1$ and $h_2$ are coprime to one another in~$\QQ[x]$. Let $\calI$ be the 
isogeny class that contains all varieties isogenous to $A_1\times A_2$ for 
$A_i\in \calI_i$, so that the real Weil polynomial of $\calI$ is $h_1 h_2$.

Suppose there is an integer $n>1$ with the following property\textup{:} For
every $A\in \calI$ that has a principal polarization, if we let 
$A_1\in\calI_1$, $A_2\in\calI_2$, and $\Delta\hookrightarrow A_1\times A_2$ be
as in Proposition~\textup{\ref{P:splitting}}, then there is an automorphism
$\alpha$ of $A_2$ such that
\begin{itemize}
\item[(a)] $\alpha$ has order $n$\textup{;}
\item[(b)] $\alpha^\dagger\alpha = 1$ for every Rosati involution on
           $(\End A_2)\otimes\QQ$\textup{;} and
\item[(c)] $\alpha - 1$ annihilates the image of $\Delta$ in~$A_2$.
\end{itemize}
Then every curve $C$ whose Jacobian has real Weil polynomial $h_1 h_2$ has an
automorphism $\varepsilon$ of order~$n$\textup{;} if $n=2$ and $C$ is 
hyperelliptic, there is an $\varepsilon$ of order $2$ that is \emph{not} the 
hyperelliptic involution.
\end{proposition}

\begin{remark}
In actual applications we are likely to use the proof of this result in addition
to its statement, because specific properties of the automorphisms $\alpha$ will
give us more information about the properties of the possible $\varepsilon$. See
for example Theorem~\ref{T:resultant2} and Section~\ref{ssec:otherauts}.
\end{remark}

\begin{proof}[Proof of Proposition~\textup{\ref{P:aut}}]
Suppose $(J,\lambda)$ is the polarized Jacobian of a curve $C$ with real Weil
polynomial $h_1 h_2$. We let $A_1$, $A_2$, and 
$\Delta\hookrightarrow A_1\times A_2$ be associated to $J$ as in 
Proposition~\ref{P:splitting}, so that there is an exact sequence
\begin{equation}
\label{EQ:deltaJ}
0 \longrightarrow \Delta \longrightarrow A_1\times A_2\longrightarrow J \longrightarrow 0.
\end{equation}
The polarization $\lambda$ on $J$ pulls back to give a polarization 
$\lambda_1\times\lambda_2$ on $A_1 \times A_2$ of degree $(\#\Delta)^2$, where 
$\#\Delta$ denotes the rank of the group scheme~$\Delta$. Let 
$\alpha \in \Aut A_2$ be as in the statement of the proposition. Note that 
$\alpha$ is an automorphism of the polarized variety $(A_2,\lambda_2)$ because
$\alpha^\dagger\alpha = 1$ for the Rosati involution associated to~$\lambda_2$.
Consider the automorphism $1 \times \alpha$ of the polarized variety 
$(A_1\times A_2, \lambda_1\times\lambda_2)$. This automorphism acts trivially on
the image of $\Delta$ in $A_1\times A_2$, so it descends to give an automorphism
$\beta$ of the polarized variety $(J,\lambda)$. Clearly $\beta$ has order~$n$.

Our comments above about the relationship between the automorphism group of a
curve and that of its Jacobian show that there is an automorphism $\varepsilon$
of $C$ such that $\varepsilon^* = \pm\beta$. If $\varepsilon^* = \beta$ then 
clearly $\varepsilon$ has order $n$ and we are done, so assume that
$\varepsilon^* = -\beta$. If $n$ is odd then $\varepsilon^2$ has order~$n$ and 
we are done, so assume that $n$ is even. Then $\varepsilon$ has the same order 
as the automorphism $\gamma\colonequals(-1)\times(-\alpha)$ of $A_1\times A_2$. 
Clearly the order of $\gamma$ is even. For every $m$ we have
$\gamma^{2m} = 1 \times \alpha^{2m}$, and the smallest $m$ with $\alpha^{2m}=1$
is~$n/2$. Therefore $\gamma$ has order~$n$, and hence $\varepsilon$ does as 
well.

Suppose~$n=2$. Since neither $1\times\alpha$ nor $(-1)\times(-\alpha)$ is equal 
to $(-1)\times(-1)$ on $A_1\times A_2$, we see that neither $\beta$ nor $-\beta$
is equal to~$-1$, so $\varepsilon^* \ne -1$ and $\varepsilon$ is not the 
hyperelliptic involution.
\end{proof}

Now we turn to some applications of Proposition~\ref{P:aut}.

%%%%%%%%%%%%%%%%%%%%%%%%%%%%%%%%%%%%%%%%%%%%%%%%%%%%%%%%%%%%%%%%%%%%%%%%%%%%%%%%
\subsection{The resultant 2 method}
\label{ssec:resultant2}

The primary application of Proposition~\ref{P:aut} is when we can take the
automorphism $\alpha$ of $A_2$ to be~$-1$. 

\begin{theorem}
\label{T:resultant2}
Suppose the real Weil polynomial $h$ of an isogeny class $\calI$ of abelian 
varieties over $\Fq$ can be written as a product $h = h_1 h_2$ in~$\ZZ[x]$,
where the reduced resultant of the radicals of $h_1$ and $h_2$ is~$2$. If the 
Jacobian of a curve $C$ lies in~$\calI$, then $C$ has a nonhyperelliptic 
involution, and the quotient of $C$ by this involution is a curve $D$ whose real
Weil polynomial is either $h_1$ or~$h_2$.
\end{theorem}

\begin{proof}
This result was proven by Howe and 
Lauter~\cite[Theorem~2.2(b), p.~176]{HoweLauter2012} based on an earlier and 
weaker version of the result~\cite[Theorem~1(b), p.~1678]{HoweLauter2003}, but 
it also follows very quickly from what we have already proven here.

Since the reduced resultant of the radicals of $h_1$ and $h_2$ is~$2$, we see
from Proposition~\ref{P:delta} that the $\Delta$ from 
Proposition~\ref{P:splitting} is killed by~$2$. Thus we can take ${\alpha=-1}$
in Proposition~\ref{P:aut}, and we find that $C$ has a nonhyperelliptic 
involution~$\varepsilon$. 

Let $J$, $A_1$, $A_2$,~$\Delta$, and $\beta$ be as in the proof of 
Proposition~\ref{P:aut}. The Jacobian of $D$ is isogenous to the reduced 
subvariety of the connected component of the kernel of the isogeny 
$\varepsilon^* - 1 = \pm\beta - 1$, and the isogeny class of this subvariety of
$J$ is the same as the isogeny class of the reduced subvariety of the connected
component of the kernel of the isogeny $\pm(1\times \alpha) - (1\times 1)$ on
$A_1\times A_2$. Since we chose $\alpha$ to be~$-1$, this isogeny is either
$0\times (-2)$ or $(-2)\times 0$, so the reduced subvariety of the connected
component of the kernel is either $A_1$ or~$A_2$. Thus, the real Weil polynomial
of $D$ is either $h_1$ or~$h_2$.
\end{proof}

We refer to this result as the ``resultant~2" method. It has been very useful
in populating the online table of known bounds 
on~$N_q(g)$~\cite{vanderGeerHoweEtAl2009}. For example, of the $26$ isogeny 
classes of $8$-dimensional abelian varieties over $\FF_4$ that Lauter's 
algorithm produces when searching for a genus-$8$ curve over $\FF_4$ with 
exactly $24$ rational points, $18$ are eliminated by the resultant~1 method, 
while $6$ are eliminated by short arguments that rely on the resultant~2 
method~\cite[\S 2]{Howe2021}. (One of the other two isogeny classes is 
eliminated by using Theorem~\ref{T:supersingular}, and the other by an argument
based on deducing the existence of a degree-$3$ map from a curve whose Jacobian
lies in the given isogeny class to an elliptic curve, as in 
Section~\ref{ssec:knowndegree} below.)

Even when the resultant~2 method does not immediately give us a proof that no
curve has a certain real Weil polynomial, it can tell us where we should look
for such a curve. For example, many of the upper bounds on $N_q(4)$ that were 
produced in earlier work of the author~\cite{Howe2012} rely on computer searches
through particular families of genus-$4$ curves. If the resultant~2 method tells
us that a genus-$4$ curve with a certain number of points must be a double cover
of an elliptic curve with a given trace, we can enumerate all of the elliptic
curves with that trace, and then enumerate all of their genus-$4$ double covers.
From the Riemann--Hurwitz formula, we find that in odd characteristic a 
genus-$4$ double cover $C$ of an elliptic curve $E$ over $\Fq$ must be ramified
over $6$ points, and can be given by an equation $z^2 = f$,  where $f$ is a
function on $E$ whose divisor is of the form
\[ 
P_1 + \cdots + P_6 + 2Q - 8\infty,
\]
where the $P_i$ are distinct geometric points on~$E$ and $Q$ is a rational point
on~$E$. By composing the double cover $C\to E$ with a translation $E\to E$, we
can also assume that $Q$ lies in a fixed set of representatives for 
$E(\Fq)/3E(\Fq)$.

For a given elliptic curve~$E$, we can loop through the possible points~$Q$, and
then do a Riemann--Roch computation to find all functions $f$ with a divisor of
the correct shape. For each $f$ we can compute the Weil polynomial of the double
cover of $E$ defined by $z^2 = f$, and check to see whether it is the Weil
polynomial we are looking for.

Similarly, if the resultant~2 method tells us that a genus-$4$ curve is a
double cover of a genus-$2$ curve with a known real Weil polynomial, we can 
enumerate all of the genus-$2$ curves with this real Weil polynomial, and then
enumerate the genus-$4$ double covers of this curve by considering what the
ramification divisor must look like and then doing appropriate Riemann--Roch
calculations to find all functions $f$ such that $z^2 = f$ has a ramification
divisor of the right shape.

Together, these two enumeration methods produced the current lower bounds on 
$N_q(4)$ for most of the odd prime powers less than 
$100$~\cite[Table~3, p.~84]{Howe2012}, and many of the upper bounds as 
well~\cite[Table~2, p.~72]{Howe2012}. In addition, values of $N_q(g)$ for some
larger $g$ have been improved by similar 
enumerations~\cite[{\S}6.3]{HoweLauter2003}.

%%%%%%%%%%%%%%%%%%%%%%%%%%%%%%%%%%%%%%%%%%%%%%%%%%%%%%%%%%%%%%%%%%%%%%%%%%%%%%%%
\subsection{Automorphisms of order greater than \texorpdfstring{$2$}{2}}
\label{ssec:otherauts}

Let us give an example of an application of Proposition~\ref{P:aut} where the 
automorphism has order~$3$. 

\begin{theorem}
\label{T:order3}
Every genus-$6$ curve $C$ over $\FF_2$ with real Weil polynomial
\[
(x^2 + x - 1)(x^4 + 5x^3 + 5x^2 - 5x - 5)
\]
has an automorphism $\varepsilon$ of order~$3$, and the quotient of $C$ by the
group $\langle\varepsilon\rangle$ is a genus-$2$ curve $D$ with real Weil 
polynomial $x^2 + x - 1$.
\end{theorem}

\begin{remark}
It turns out that the only genus-$2$ curve over $\FF_2$ with real Weil 
polynomial $x^2 + x - 1$ is the curve $D$ given by
\[
v^2 + (u^3 + u + 1)v + u(u^3 + u + 1) = 0.
\]
Suppose $C$ is a curve as in the statement of the theorem. The Riemann--Hurwitz 
formula shows that the Galois triple cover $C\to D$ must be ramified at $2$
geometric points of~$D$, and by looking at the number of low-degree places on
$C$ predicted by its Weil polynomial, we find that these ramification points
must actually constitute a degree-$2$ place on~$D$. The author used the ideas
of Chapman~\cite[\S 2]{Chapman1996} to search for Galois triple covers of $D$
ramified at a single degree-$2$ place, and found an example: The triple cover of
$D$ defined by
\begin{multline*}
w^3 + [(u^2 + 1)v + (u^6 + u^5 + u^4 + 1)]w \\
    + [(u^6 + u^5 + u^4 + u^3 + u)v + (u^8 + u^6 + u^4 + u)] = 0
\end{multline*}
is a genus-$6$ curve whose real Weil polynomial is the sextic given in the
statement of the theorem.
\end{remark}

\begin{remark}
The factors of the Weil polynomial in Theorem~\ref{T:order3} appear in an
isogeny class that arises in the search for genus-$12$ curves over $\FF_2$ with
$15$ points~\cite[Theorem~7.1, p.~203]{HoweLauter2012}. The application of 
Proposition~\ref{P:aut} with $n=3$ to that isogeny class shows that it contains
no Jacobians.
\end{remark}

\begin{proof}[Proof of Theorem~\textup{\ref{T:order3}}]
Let $h_1 = (x^2 + x - 1)$ and $h_2 = (x^4 + 5x^3 + 5x^2 - 5x - 5)$, let $f_1$
and $f_2$ be the corresponding Weil polynomials, let $K_1$ and $K_2$ be the 
CM~fields corresponding to $f_1$ and~$f_2$, and let $\pi_1$ and $\pi_2$ be
elements of $K_1$ and $K_2$ that satisfy $f_1$ and~$f_2$. The number field 
$K_2$ is in fact the $15$\up{th} cyclotomic field.

Using a formula for the discriminant of orders generated by Frobenius and 
Verschiebung~\cite[Proposition~9.4, p.~2384]{Howe1995}, we check that 
$\calO_1 \colonequals \ZZ[\pi_1,\pibar_1]$ is the maximal order of $K_1$ and 
$\calO_2 \colonequals \ZZ[\pi_2,\pibar_2]$ is the maximal order of~$K_2$. Both
orders have class number~$1$, so there is a unique abelian surface $A_1$ over
$\FF_2$ with real Weil polynomial~$h_1$ and a unique abelian four-fold $A_2$ 
over $\FF_2$ with real Weil polynomial~$h_2$.

The reduced resultant of $h_1$ and $h_2$ is~$3$, so by 
Propositions~\ref{P:splitting} and~\ref{P:delta} if there is a curve $C$ with
real Weil polynomial $h_1 h_2$ then there is an exact sequence
\[
0 \longrightarrow \Delta \longrightarrow A_1\times A_2\longrightarrow \Jac C \longrightarrow 0
\]
where $\Delta$ can be embedded in both $A_1[3]$ and~$A_2[3]$. 

As in the proof of Theorem~\ref{T:supersingular}, we can understand the finite
group scheme $A_1[3]$ in terms of an $\calO_1$-module. The Deligne module of
$A_1$ is isomorphic to $\calO_1$ because $\calO_1$ has class number~$1$, so the 
module associated to  $A_1[3]$ is simply $\calO_1/3\calO_1$, and since the
rational prime $3$ is inert in~$K_1$, the module $\calO_1/3\calO_1$ has no 
nontrivial submodules. Likewise, the $\calO_2$-module associated to $A_2[3]$ is 
$\calO_2/3\calO_2$. If we let $\zeta$ be one of the cube roots of unity 
in~$K_2$, then the rational prime $3$ decomposes in $\calO_2$ as $\frakp^2$,
where $\frakp$ is the prime ideal generated by $\zeta - 1$. Therefore the module
$\calO_2/3\calO_2$ has exactly one nontrivial submodule, whose order is~$81$.

The group scheme $\Delta$ cannot be trivial, because every polarization of 
$A_1\times A_2$ is decomposable. On the other hand, as we have just seen, the 
rank-$81$ group scheme $A_1[3]$ has no nontrivial subgroup schemes. Therefore 
$\Delta$ must be isomorphic to~$A_1[3]$. The rank-$3^8$ group scheme $A_2[3]$ 
has one nontrivial subgroup scheme, whose rank is~$81$, and it must be the image
of $\Delta$ in~$A_2$. The module associated to the rank-$81$ subgroup scheme of
$A_2[3]$ is killed by $\zeta - 1$, so $\zeta - 1$ annihilates the image of 
$\Delta$ in~$A_2$.

From Proposition~\ref{P:aut} we see that $C$ has an automorphism $\varepsilon$
of order~$3$. From the proposition's proof we see that in fact we can choose 
$\varepsilon$ so that the automorphism $\varepsilon^*$ of $\Jac C$, pulled back
to $A_1\times A_2$, is the automorphism $\gamma\colonequals 1\times\zeta$. The 
reduced subvariety of the connected component of the kernel of $\gamma - 1$ 
is~$A_1$, so the Weil polynomial of the quotient of $C$ by 
$\langle\varepsilon\rangle$ is~$h_1$.
\end{proof}

%%%%%%%%%%%%%%%%%%%%%%%%%%%%%%%%%%%%%%%%%%%%%%%%%%%%%%%%%%%%%%%%%%%%%%%%%%%%%%%%
\subsection{Automorphisms in nonsplit isogeny classes}
\label{ssec:nonsplit}

Proposition~\textup{\ref{P:aut}} only applies to isogeny classes whose real Weil
polynomials have more than one irreducible factor. There are also circumstances
under which one can prove that if a curve has a Jacobian isogenous to a power of
a \emph{simple} variety, then the curve must have a nontrivial automorphism.

\begin{theorem}
\label{T:cyclotomic}
Let $f = f_0^e$ be the Weil polynomial of a simple isogeny class of abelian 
varieties over~$\Fq$, where $f_0$ is irreducible. Let $K$ be the number field 
defined by~$f_0$, let $\pi$ be a root of $f_0$ in~$K$, and let $\pibar = q/\pi$.
If the ring $\ZZ[\pi,\pibar]$ contains a root of unity~$\zeta$, then every curve
over $\Fq$ whose Weil polynomial is a power of $f$ has an automorphism 
$\varepsilon$ such that $\varepsilon^* = \pm\zeta$, where the sign can be taken
to be $+1$ if the curve is hyperelliptic.
\end{theorem}

\begin{proof}
Suppose $C$ is a curve whose Weil polynomial is a power of~$f$, and let $F$ and
$V$ be the Frobenius and Verschiebung on the polarized Jacobian $(J,\lambda)$
of~$C$. The subring $\ZZ[F,V]$ of $\End J$ lies in the center of $\End J$, and
this ring is isomorphic to $\ZZ[\pi,\pibar]$, so we may view $\zeta$ as an
automorphism of~$J$. As we noted in our discussion of the Rosati involution in 
Section~\ref{ssec:automorphisms}, the Rosati involution on $(\End J)\otimes \QQ$
associated to $\lambda$ restricts to complex conjugation on~$K$, so $\zeta$
satisfies $\zeta^\dagger\zeta = 1$, and this means that $\zeta$ is an 
automorphism of $(J,\lambda)$. By Proposition~\ref{P:torelli}, $C$ has an 
automorphism $\varepsilon$ such that $\varepsilon^* = \pm \zeta$, and we can 
take the sign to be positive if $C$ is hyperelliptic.
\end{proof}

\begin{example}
Consider the real Weil polynomial $h = x^3 + 4 x^2 + 3x - 1$ over~$\FF_2$, and 
let $f$ be the corresponding Weil polynomial. Let $K$ be the $7$\up{th}
cyclotomic field and let $\zeta\in K$ be a primitive $7$\up{th} root of unity. 
We check that $\pi \colonequals \zeta^5 + \zeta^3 + \zeta^2 + \zeta$ is a root
of~$f$, and that $\zeta =   \pi^3 + 2\pi^2 + 4\pi + 5 + 2\pibar + \pibar^2.$ 
Theorem~\ref{T:cyclotomic} shows that every curve $C$ over $\FF_2$ with real 
Weil polynomial $h$ must have an automorphism $\varepsilon$ of order~$7$.

There is such a genus-$3$ curve over $\FF_2$. It is a twist of the Klein 
quartic, and a model can be created as follows: Choose an element $M$ of
$G \colonequals \PGL_3(\FF_2)$ of order~$7$. The group $G$ acts on the set $L$
of the seven nonzero linear forms in three variables over~$\FF_2$. Let $f$ be 
the sum over all $p\in L$ of $p^3\cdot M(p)$. Then $f$ is a nonzero quartic form
that defines a curve $C$ in $\PP^2_{\FF_2}$. One can check that this curve $C$
has real Weil polynomial~$h$.
\end{example}

Next we give another example of an application of Theorem~\ref{T:cyclotomic} 
that works over every finite field. This result was first conjectured by Maisner
and Nart~\cite[\S 4.1]{MaisnerNartEtAl2002}, and was proven in the appendix of 
that paper.

\begin{theorem}
\label{T:MN1}
Let $q$ be a prime power and let $\calI_q$ be the isogeny class of abelian 
surfaces over $\Fq$ with Weil polynomial $f = x^4 + (1-2q)x^2 + q^2$. Then there
are no Jacobians in~$\calI_q$.
\end{theorem}

\begin{proof}
A proof is given in~\cite[Appendix]{MaisnerNartEtAl2002}; we give a modified
version here.

Suppose, to get a contradiction, that there is a curve $C$ over $\Fq$ with $f$
as its Weil polynomial. Let $K$ be the quartic number field obtained by 
adjoining to $\QQ$ a square root $s$ of $4q-1$ and a square root $i$ of~$-1$.
If we set $\pi = (s + i)/2$ then we check that $\pi$ is a root of~$f$. We have 
$\pibar = (s - i)/2$, and so $i = \pi-\pibar$. According to 
Theorem~\ref{T:cyclotomic}, $C$ must have an automorphism $\varepsilon$ with
$\varepsilon^* = i$. Note that then $(\varepsilon^2)^* = -1$, so that
$\varepsilon^2$ is the hyperelliptic involution. 

First consider the case where $q$ is odd. Let $C\to D$ be the cyclic degree-$4$
Galois extension with group $\langle\varepsilon\rangle$, and note that 
$D\cong\PP^1$ because $C\to D$ factors through the hyperelliptic map 
$C\to\PP^1$. Let $m$ and $n$  denote the number of geometric points on $D$ with
ramification group $\ZZ/2\ZZ$ and $\ZZ/4\ZZ$, respectively. The Riemann--Hurwitz
formula shows that $10 = 2m + 3n$, and we must have ${n>0}$ because otherwise 
the intermediate double cover $\PP^1\to D$ would be unramified. The only 
possibility is that $m = n = 2$, and it follows that all of the points of $D$
that ramify in $C\to D$ are rational over $\FF_{q^2}$, and therefore all of the
points of $C$ that ramify in $C\to D$ are rational over~$\FF_{q^4}$.

Consider the points of $C(\FF_{q^4})$. The group $\langle\varepsilon\rangle$
acts on these points, and for the points $P$ of $C$ that do not ramify in 
$C\to D$, the orbit of $P$ under this action has size~$4$. The $n = 2$ totally 
ramified points have orbits of size~$1$, and the $2m = 4$ partially ramified 
points have orbits of size~$2$. Thus, we must have
$\#C(\FF_{q^4})\equiv 2\bmod 4$. However, the Weil polynomial predicts that 
\[
\#C(\FF_{q^4}) = q^4 + 1 - \Tr(\pi^4) = q^4 - 4q^2 + 8q - 1 \equiv 0\mod 4.
\]
This contradiction proves the theorem when $q$ is odd.

Now consider the case where $q$ is a power of~$2$. We note that the middle 
coefficient of the Weil polynomial of $C$ is odd, so $C$ is ordinary. 
Lemma~\ref{L:order4} below then completes the proof.
\end{proof}

\begin{lemma}
\label{L:order4}
No ordinary genus-$2$ curve over $k = \overline{\FF}_2$ has an automorphism
whose square is the hyperelliptic involution.
\end{lemma}

\begin{proof}
Suppose $C$ were such a curve, with $\alpha$ being an automorphism whose square 
is the hyperelliptic involution~$\iota$. We can write $C$ as $y^2 + y = f$ for a
rational function $f\in k(x)$ with no poles of even order, and since $C$ is 
ordinary all the poles of $f$ must be
simple~\cite[Proposition~3.1, p.~8]{Subrao1975}. The Riemann--Hurwitz formula 
then shows that $f$ must have three poles, each of which corresponds to a 
Weierstrass point on~$C$. The automorphism $\alpha$ induces an involution
$\beta$ on the curve $C/\langle\iota\rangle$ whose function field is generated 
by~$x$, and by changing coordinates we can assume that $\beta$ is given by 
$x\mapsto x+1$. 

The automorphism $\alpha$ cannot fix all three Weierstrass points on $C$ because
then $\beta$ would fix the three $x$-coordinates of these points, while $\beta$ 
has only one fixed point on~$\PP^1$. Therefore $\alpha$ must fix one of the 
Weierstrass points and swap the other two. By changing coordinates in $x$ again,
we can assume that the Weierstrass point that is fixed has 
$x$-coordinate~$\infty$, while the two that are swapped have $x$-coordinates $0$
and~$1$.

Now our curve $C$ is of the form $y^2 + y = ax + b/x + b/(x+1)$. The 
automorphism $\alpha$ sends $x$ to~$x+1$, so it must send $y$ to $y + c$ where
$c^2 + c = a$. But the automorphism $(x,y)\to(x+1,y+c)$ has order~$2$, not~$4$, 
and we have reached a contradiction.
\end{proof}

\begin{remark}
Similar reasoning shows that no almost-ordinary genus-$2$ curve over 
$\overline{\FF}_2$ has an automorphism whose square is the hyperelliptic
involution. On the other hand, every supersingular genus-$2$ curve over 
$\overline{\FF}_2$ has exactly $20$ automorphisms that square to the 
hyperelliptic involution~\cite[Lemma~2.1]{BrockHowe2020}. In principle, 
Lemma~\ref{L:order4} and the comments we have just made follow from Igusa's
classification~\cite[pp.~645--46]{Igusa1960} of the possible automorphism 
groups of genus-$2$ curves in characteristic~$2$, but Igusa only lists the 
\emph{reduced} automorphism groups, so some additional computation is needed in 
any case.
\end{remark}

%%%%%%%%%%%%%%%%%%%%%%%%%%%%%%%%%%%%%%%%%%%%%%%%%%%%%%%%%%%%%%%%%%%%%%%%%%%%%%%%
\subsection{Deducing maps of known degree to elliptic curves}
\label{ssec:knowndegree}

If an elliptic curve $E$ appears as an isogeny factor of the Jacobian of a 
curve~$C$ over~$\Fq$, then there is guaranteed to be a nonconstant map from $C$ 
to $E$: Such a map can be obtained by as the composition of an Abel--Jacobi map
from $C$ into its Jacobian, followed by a nonconstant map from the Jacobian 
to~$E$. The mere existence of such a map, without any further information, does
not necessarily make it easier to determine any more useful information 
about~$C$. But if one knows more specific information --- for instance, if one
knows that $C$ has a map of a specific degree $d$ to an elliptic curve isogenous
to $E$ --- then there are more things one can say. As a simple example, the
existence of a degree-$d$ map from $C$ to $E$ implies that
$\#C(\Fq) \le d \#E(\Fq)$.

The problem is, how can we determine bounds on the degree? In this section we 
will show how keeping track of polarizations will let us determine the possible 
degrees of maps to elliptic curves. The following basic result is the key.

\begin{proposition}
\label{P:pullback}
Let $(J,\lambda)$ be the polarized Jacobian of a curve~$C/\Fq$, let $D$ be a 
degree-$1$ divisor on~$C$, and let $\varepsilon\colon C\to J$ be the 
Abel--Jacobi embedding that sends a point $P\in C(\Fbar_q)$ to the class of the
divisor~$P-D$. Suppose there is a nonzero homomorphism $\psi\colon E\to J$ from 
an elliptic curve $E/\Fq$ to~$J$. Let $\mu$ be the unique principal polarization
on~$E$, so that the pullback $\psi^*\lambda = \psihat\lambda\psi$ of 
$\lambda$ is equal to $d\mu$ for some integer~$d>0$. Then the composition 
$\varphi\colonequals\mu^{-1}\psihat\lambda\varepsilon$ is a degree-$d$ map
from $C$ to~$E$.
\end{proposition}

\begin{proof}
This is \cite[Lemma~4.3, p.~185]{HoweLauter2012}. We present the short proof 
here because the result is basic to the ideas in this section.

The diagram
\[
\xymatrix@C=3em{
   & E\ar[r]^{d}\ar[d]_{\psi}  & E\ar[r]^{\mu}_{\sim}               & \Ehat \\
C\ar[r]^-{\varepsilon}\ar`d[r]`[rrr]+/r 3em/`[u]+/u 2em/_{\varphi}`[rr][rru]
   & J\ar[rr]^{\lambda}_{\sim} &                                    & \Jhat\ar[u]_{\psihat} 
}
\]
summarizes the relationships among the maps mentioned in the theorem. We obtain 
a slightly expanded diagram
\[
\xymatrix@C=3em{
   & E\ar[r]^{d}\ar[d]_{\psi}  & E\ar[r]^{\mu}_{\sim}           & \Ehat \\
C\ar[r]^-{\varepsilon}\ar`d[r]`[rrr]+/r 3em/`[u]+/u 2em/_{\varphi}`[rr][rru]
   & J\ar[r]^{1}               & J\ar[r]^{\lambda}_{\sim}\ar[u] & \Jhat\ar[u]_{\psihat} 
}
\]
by inserting a second copy of $J$ in the bottom row. The middle vertical map is
clearly equal to the pushforward $\varphi_*\colon J\to E$, which sends the class
of a degree-$0$ divisor on $C$ to the class of its image under $\varphi$. Now,
$\varphi^*$ and $\varphi_*$ are dual to one another, in the sense that the
diagram
\[
\xymatrix@C=5em{
E\ar[r]^{\mu}_{\sim}                       & \Ehat \\
J\ar[r]^{\lambda}_{\sim}\ar[u]_{\varphi_*} & \Jhat\ar[u]_{\varphistarhat} 
}
\]
commutes~\cite[Lemma~4.4, p.~186]{HoweLauter2012}. Comparing this diagram to the
right-hand square of the preceding diagram, we see that 
$\psihat - \varphistarhat = 0$, so we must have $\psi = \varphi^*$.
Therefore $\varphi_*\varphi^* = \varphi_*\psi = d$, so $\varphi$ has degree~$d$.
\end{proof}

Suppose $t$ is the trace of Frobenius of an elliptic curve over~$\Fq$. In the 
case of real Weil polynomials that have a single factor of~$x-t$, the preceding 
result, combined with Propositions~\ref{P:splitting} and~\ref{P:delta}, already
gives us a useful result.

\begin{theorem}
\label{T:factorE}
Suppose the real Weil polynomial $h\in \ZZ[x]$ of a curve $C$ over $\Fq$ can be 
written as a product $h = h_0 \cdot (x - t)$, where $t$ is the trace of an 
elliptic curve over $\Fq$ and where $h_0$ is coprime to~$x-t$. Let $g_0$ be the 
radical of $h_0$ and let $r = \lvert g_0(t)\rvert$. Then $C$ has a nonconstant
map of degree dividing $r$ to an elliptic curve with trace~$t$.
\end{theorem}

\begin{proof}
This is~\cite[Proposition~2.5, p.~177]{HoweLauter2012}, but the theorem follows
almost immediately from the results mentioned above. If we let $(J,\lambda)$ be 
the polarized Jacobian of $C$, then Proposition~\ref{P:splitting} gives us a 
diagram
\[
\xymatrix@C=5em{
A\times E\ar[r]^{\lambda_1\times\lambda_2}\ar[d] & \Ahat\times \Ehat\\
J\ar[r]^\lambda                                  & \Jhat\ar[u]
}
\]
where $\lambda_1$ and $\lambda_2$ are polarizations whose kernels are isomorphic
to the same group scheme~$\Delta$, and Proposition~\ref{P:delta} tells us that 
$\Delta$ is killed by~$r$. If we let $\mu$ be the principal polarization of~$E$,
it follows that~$\lambda_2$, which is the pullback of $\lambda$ to~$E$, is equal
to $d\mu$ for some $d$ that divides~$r$. The theorem then follows from 
Proposition~\ref{P:pullback}.
\end{proof}

\begin{example}
Rigato~\cite{Rigato2010a} shows that there is no genus-$4$ curve over $\FF_7$ 
having $25$ rational points. Her proof begins by showing that if $C$ is such a 
curve, the only possibility for its real Weil polynomial is  $h = (x+2)(x+5)^3$,
because the other options produced by Lauter's algorithm are eliminated by the 
resultant~1 method. Theorem~\ref{T:factorE}, applied to this real Weil 
polynomial, shows that $C$ must have a map of degree~$3$ to an elliptic curve 
over $\FF_7$ with $10$ points. Rigato then shows that no such degree-$3$ cover
is consistent with the possible point counts on the curves; the Galois closure 
of the extension $C\to E$ would have impossible properties. 
See~\cite{Rigato2010a} for details.
\end{example}

Theorem~\ref{T:factorE} can be used to show that every genus-$6$ curve over
$\FF_2$ with real Weil polynomial $(x-1)(x+2)(x^2 + 3x + 1)^2$ has a degree-$3$
map to the elliptic curve over $\FF_2$ with trace~$1$, and by analyzing such 
triple covers Rigato~\cite[\S 8]{Rigato2010b} shows that there is a unique curve
with the given real Weil polynomial. (It is one of the two genus-$6$ curves over
$\FF_2$ with exactly $N_2(6)$ points.) Likewise, Theorem~\ref{T:factorE} is used
in~\cite{Howe2021} to show that a genus-$8$ curve over $\FF_4$ with real Weil
polynomial $x(x+2)^4(x+3)(x+4)^2$ is a triple cover of the unique elliptic curve
over $\FF_4$ with $8$ points, and an analysis similar to Rigato's shows that no
such triple cover exists; this is part of the proof that $N_4(8)\le 23$.

Theorem~\ref{T:factorE} applies only to isogeny classes whose isogeny factors
include a \emph{single} copy of an elliptic curve of trace~$t$. There is a
related result that holds for isogeny classes whose isogeny factors include a
\emph{power} of an ordinary elliptic curve, but the conclusion of this result is
most naturally expressed in terms of an upper bound on the degree of the map,
instead of limiting the degree to a set of divisors.

\begin{theorem}
\label{T:factorEn}
Suppose the real Weil polynomial $h\in \ZZ[x]$ of a curve $C$ over $\Fq$ can be 
written as a product $h = h_0 \cdot (x - t)^n$ for some~$n>0$, where $t$ is the 
trace of an ordinary elliptic curve over $\Fq$ and where $h_0$ is coprime 
to~$x-t$. Let $g_0$ be the radical of~$h_0$, let $r = \lvert g_0(t)\rvert$, and
let $b = \gcd(r^n,h_0(t))$. Let $E$ be any elliptic curve over $\Fq$ with trace
$t$ whose endomorphism ring is generated over $\ZZ$ by the Frobenius. Then there
is a nonconstant map from $C$ to $E$ of degree at most
\[ 
\gamma_{2n}\, b^{1/n} \sqrt{|t^2 - 4q|/4},
\]
where $\gamma_{2n}$ is the Hermite constant for dimension~$2n$.
\end{theorem}

\begin{proof}
This is~\cite[Proposition~4.1, p.~183]{HoweLauter2012}, except that we replace 
the ``gluing exponent" in the statement of that theorem with the reduced 
resultant of $h_0$ and~$x-t$, which is~$\lvert g_0(t)\rvert$. This is allowable,
because the gluing exponent divides this reduced 
resultant~\cite[Proposition~2.8, p.~178]{HoweLauter2012}.
\end{proof}

\begin{remark}
It is known that 
\[
    \gamma_2^2 = 4/3,  \quad
    \gamma_4^4 = 4,    \quad 
    \gamma_6^6 = 64/3, \quad 
    \gamma_8^8 = 256,  \quad \text{and\ }
    \gamma_{10}^{10} < 5669.
\]
General upper bounds for $\gamma_n$ are known as 
well~\cite[\S38]{GruberLekkerkerker1987}.
\end{remark}

\begin{example}
\label{ex:2-12-15}
As of this writing, the best upper bound we have for the number of points on a 
genus-$12$ curve over $\FF_2$ is~$15$, but no genus-$12$ curve with $15$ points
is known. There are only three possibilities for the real Weil polynomial of 
such a curve~\cite[Theorem~7.1, p.~203]{HoweLauter2012}, and one of them is 
$h = h_0 \cdot (x + 1)^2$, where 
\[
h_0 = (x + 2)^2 (x^2 - 2) (x^2 + 2x - 2)^3.
\]
Let us apply Theorem~\ref{T:factorEn} to this real Weil polynomial.

We compute that $r = 3$ and that $b = \gcd(3^2,-3^3) = 9$. Let $E$ be the unique
elliptic curve over $\FF_2$ with trace~$-1$. The endomorphism ring of $E$ is 
isomorphic to the ring of integer $\calO$ of $\QQ(\sqrt{-7})$, and Frobenius 
generates the endomorphism ring. If $C$ is a curve whose real Weil polynomial 
is~$h$, then Theorem~\ref{T:factorEn} shows that there is a map from $C$ to $E$
whose degree is at most
\[
\gamma_4 \, 9^{1/2} \sqrt{7/4} < 6.
\]
Now, since $\#E(\FF_2) = 4$ and $\#C(\FF_2) = 15$, there cannot be a map from 
$C$ to $E$ of degree $3$ or less; if every rational point of $E$ were to split
completely in such a cover, that would still not account for all of the rational
points on~$C$. Therefore, every curve $C$ with real Weil polynomial $h$ must
have a map of degree $4$ or $5$ to~$E$.

In fact, a closer analysis of the  Hermitian forms on $\calO\times \calO$ shows
that the conclusion of Theorem~\ref{T:factorEn} can be sharpened in this
example, and that a degree-$5$ map from $C$ to $E$ is not possible. 
(See~\cite[\S 4]{HoweLauter2012} for a discussion of this.) Thus, a curve $C$ 
with the given real Weil polynomial must have a map of degree $4$ to~$E$.
\end{example}

%%%%%%%%%%%%%%%%%%%%%%%%%%%%%%%%%%%%%%%%%%%%%%%%%%%%%%%%%%%%%%%%%%%%%%%%%%%%%%%%
\subsection{Galois descent}
\label{ssec:descent}

Let $\calI$ be an isogeny class of abelian varieties over~$\Fq$, where 
$q = q_0^e$ for a prime power $q_0$ and an integer~${e>1}$. Sometimes we can
show that every variety in $\calI$ is the base extension of a variety defined 
over $\FF_{q_0}$. The following theorem gives us one situation where this 
occurs; to state the theorem we need to introduce some notation.

Let $f\in\QQ[x]$ be the Weil polynomial for $\calI$ and let $g$ be the radical
of~$f$. Frobenius acts semisimply on each variety in $\calI$, so $g$ is in fact
the minimal polynomial of Frobenius, and for each $A\in\calI$ the subring of
$(\End A)\otimes\QQ$ generated over $\QQ$ by the Frobenius is isomorphic to the
product of fields $\QQ[x]/(g)$. The subring of $\QQ[x]/(g)$ generated by 
Frobenius and Verschiebung depends only on $\calI$ and not on $A$; we denote 
this ring by $R_\calI$, and write $\pi_\calI$ and $\pibar_\calI$ for the
Frobenius and Verschiebung in $R_\calI$. By definition we have 
$R_\calI = \ZZ[\pi_\calI,\pibar_\calI]$.

\begin{theorem}
\label{T:descendA}
Suppose $\calI$ is an isogeny class of ordinary abelian varieties over~$\Fq$, 
where $q = q_0^e$. Suppose there is an element $\pi_0\in R_\calI$ such that
$\pi_\calI = \pi_0^e$, and let $f_0$ be the characteristic polynomial of $\pi_0$
acting on the varieties in~$\calI$. Then $f_0$ is the Weil polynomial of an 
isogeny class $\calI_0$ over $\FF_{q_0}$, and base extension from $\FF_{q_0}$ to
$\FF_q$ gives an equivalence between the category of abelian varieties in 
$\calI_0$ and the category of abelian varieties in~$\calI$.
\end{theorem}

\begin{remark}
There are versions of Theorem~\ref{T:descendA} that work for nonordinary 
varieties as well~\cite[Appendix]{LauterSerre2001}, but the statement is much
simpler for ordinary isogeny classes.
\end{remark}

\begin{proof}[Proof of Theorem~\textup{\ref{T:descendA}}]
This is part of the statement of~\cite[Theorem~5.2, p.~196]{HoweLauter2012}.
Here we will simply sketch the proof in the case where $\calI$ is simple.

Let $K\colonequals R_\calI\otimes\QQ$, so that $K$ is a CM field. By Deligne's
equivalence of categories~\cite{Deligne1969}, the category of abelian varieties
in $\calI$ is equivalent to the category of fractional $R_\calI$-ideals in~$K$. 

Since $\pi_0\pibar_0$ is a totally positive element of $K$ and 
$(\pi_0\pibar_0)^e = \pi_\calI\pibar_\calI = q = q_0^e$, we see that 
$\pi_0\pibar_0 = q_0$, so $\pi_0$ is a Weil number over~$\FF_{q_0}$, 
corresponding to a simple isogeny class $\calI_0$. The condition that $\calI$ 
is ordinary implies that $\pi_\calI$ and $\pibar_\calI$ are coprime 
in~$R_\calI$, because the ideal generated by these two elements contains both 
$q$ and the middle coefficient of~$f$, which is coprime to~$q$. Therefore 
$\pi_0$ and $\pibar_0$ are also coprime, and hence $\calI_0$ is also ordinary.

Deligne's equivalence of categories now tells us that the category of abelian 
varieties in $\calI_0$ is equivalent to the category of fractional 
$\ZZ[\pi_0,\pibar_0]$-ideals in~$K$. But the two orders $\ZZ[\pi_0,\pibar_0]$ 
and $R_\calI = \ZZ[\pi_\calI,\pibar_\calI]$ are equal, because
$\pi_0,\pibar_0\in R_\calI$ and because the two elements $\pi_\calI = \pi_0^e$
and $\pibar_\calI = \pibar_0^e$ belong to $\ZZ[\pi_0,\pibar_0]$. Therefore the
categories $\calI_0$ and $\calI$ are equivalent.

Tracing through Deligne's construction, we find that one equivalence between 
$\calI_0$ and $\calI$ is given by base extension.
\end{proof}

\begin{theorem}
\label{T:descendC}
Suppose $\calI$ is an isogeny class of $g$-dimensional ordinary abelian 
varieties over~$\Fq$, where $q = q_0^e$, and suppose there is an element 
$\pi_0\in R_\calI$ such that $\pi_\calI = \pi_0^e$. Let $f_0\in\ZZ[x]$ be the
characteristic polynomial of $\pi_0$ acting on the varieties in~$\calI$. 

If $C$ is a curve over $\Fq$ whose Jacobian lies in $\calI$, then $C$ has a 
model defined over $\FF_{q_0}$ whose Weil polynomial is either $f_0$ or 
$(-1)^g f_0(-x)$. If $e$ is odd, the model can be chosen with Weil 
polynomial~$f_0$.
\end{theorem}

\begin{proof}
See~\cite[Théorème~0, p.~35]{LauterSerre2001}, 
\cite[Theorem~5.3, p.198]{HoweLauter2012}.
\end{proof}

\begin{example}
Consider the isogeny class $\calI$ of abelian varieties over $\FF_{32}$ with
real Weil polynomial  $(x + 11)^3 (x^2 + 19x + 87)$. The radical of the Weil 
polynomial of this isogeny class is
\[
g\colonequals (x^2 + 11 x + 32)(x^4 + 19 x^3 + 151 x^2 + 608 x + 1024),
\]
and if we let $\pi$ be a root of $g$ in $R_\calI\otimes \QQ$ we can check that
$\pi = \pi_0^5$, where
\[
\pi_0\colonequals 4 \pi^3 + 99 \pi^2 + 1046 \pi + 5976 + 578 \pibar + 24 \pibar^2.
\]
The characteristic polynomial of $\pi_0$ is
\[
(x^2 + x + 2)^3(x^4 - x^3 + x^2 - 2x + 4),
\]
and using Theorem~\ref{T:descendC} we find that if $C$ is a curve over
$\FF_{32}$ with Jacobian in~$\calI$, then $C$ is the base extension of a curve 
$C_0$ over $\FF_2$ with real Weil polynomial $(x + 1)^3(x^2 - x - 3)$. But we
check that a curve with this real Weil polynomial would have to have a negative 
number of points over $\FF_8$, and therefore no such~$C_0$ --- and no
such~$C$ --- can exist.

The isogeny class $\calI$ arises when one looks for a genus-$5$ curve over 
$\FF_{32}$ with $85$ points. Unfortunately, there is another isogeny class that 
also arises in this case that we cannot yet eliminate, corresponding to the real
Weil polynomial $(x + 8) (x + 11)^4$, so at present the most we can say about 
$N_{32}(5)$ is that it is either $83$, $84$, or~$85$.
\end{example}

\begin{remark}
One example of a descent argument for a nonordinary isogeny class is given by
Howe and Lauter~\cite[{\S}5]{HoweLauter2003} as part of their proof that 
$N_8(5) \le 30$. One can see there some of the issues that arise in the 
nonordinary case.
\end{remark}

%%%%%%%%%%%%%%%%%%%%%%%%%%%%%%%%%%%%%%%%%%%%%%%%%%%%%%%%%%%%%%%%%%%%%%%%%%%%%%%%
%%%%%%%%%%%%%%%%%%%%%%%%%%%%%%%%%%%%%%%%%%%%%%%%%%%%%%%%%%%%%%%%%%%%%%%%%%%%%%%%
\section{Conclusion and prospects}
\label{sec:conclusion}

As our discussion in this chapter makes clear, there are many specialized tools
at our disposal to analyze isogeny classes of abelian varieties over finite
fields to see what we can say about the Jacobians they may or may not contain.
Most of these tools apply only to isogeny classes of special forms, and we are 
lucky that curves with many points often lie in isogeny classes for which the
tools are helpful. Automation is critical for this analysis; most of the recent 
improvements in upper bounds for $N_q(g)$ for specific $q$ and $g$ have involved
analyzing dozens of isogeny classes, each of which requires consideration of 
many possible splittings and many possible techniques.

There are some relatively common situations whose analysis has not yet been
automated. For example, there have been several cases arising ``in nature"
where it was necessary to derive information from the existence of a degree-$3$
map from a curve with a given real Weil polynomial to an elliptic 
curve~\cite{Howe2021,Rigato2010a,Rigato2010b}, and in each of these cases the 
analysis was done by hand. And yet, further automation will only take us so far,
because our toolkit is limited. As far as the author is aware, if we are
presented with an absolutely simple ordinary isogeny class of odd dimension 
greater than~$3$ over a prime field, and if its Weil polynomial defines a number
field containing no nontrivial roots of unity, and if the Weil polynomial
predicts a non-negative number of places of every degree, and if the place 
counts are allowed by the Oesterlé bound~\cite[Chapter~6]{Serre2020} and the 
Stohr--Voloch bound~\cite[Propositions~3.1 and~3.2, p.~15]{StohrVoloch1986},
then there are geometrically indecomposable principally polarized varieties in 
the isogeny class, but none of our current techniques will give us any further
information about curves whose Jacobians lie in the isogeny class.

In 1981, Manin~\cite{Manin1981} asked a basic question that provided some of the
initial motivation for much of the work on $N_q(g)$: ``What is the maximum 
number of points on a curve [of a given genus] over $\FF_2$?" Despite the vast
progress that has been made in filling in the online tables of bounds on 
$N_q(g)$~\cite{vanderGeerHoweEtAl2009}, when we look back at Manin's specific
question, we find that we know the exact value of $N_2(g)$ for only $21$ of the
genera between $1$ and~$50$. The first genus for which the value is unknown is
$g=12$ (see Example~\ref{ex:2-12-15}), leaving us with a simply stated question:
Is there a genus-$12$ curve over $\FF_2$ with $15$ points? 

What new ideas will we need in order to bring the answer to this question within
reach?

%%%%%%%%%%%%%%%%%%%%%%%%%%%%%%%%%%%%%%%%%%%%%%%%%%%%%%%%%%%%%%%%%%%%%%%%%%%%%%%%
%%%%%%%%%%%%%%%%%%%%%%%%%%%%%%%%%%%%%%%%%%%%%%%%%%%%%%%%%%%%%%%%%%%%%%%%%%%%%%%%

\nocite{DiPippoHowe2000}
\nocite{HoweLauter2007}

\bibliography{deducing}
\bibliographystyle{hplaindoi}

\end{document}